\documentclass[reqno]{amsart}
\usepackage{amssymb,booktabs,cite,enumitem,array}
\usepackage[unicode,hidelinks,hypertexnames=false]{hyperref}
\makeatletter
\@namedef{subjclassname@2020}{\textup{2020} Mathematics Subject Classification}
\makeatother
\theoremstyle{plain}
\newtheorem{thm}{Theorem}[section]
\newtheorem{lem}[thm]{Lemma}
\newtheorem{cor}[thm]{Corollary}
\theoremstyle{definition}
\newtheorem{remark}[thm]{Remark}
\newcommand{\V}{\mathsf{V}}
\newcommand{\FB}{\operatorname{FB}}
\newcommand{\Id}{\operatorname{Id}}
\newcommand{\AI}{\mathrm{AI}}

\let\leq\leqslant
\let\geq\geqslant
\begin{document}

\title[Finite bases and two-element factors]{Two-element factors and finite bases for additively idempotent semirings}
\author[QIRUI MA]{QIRUI MA}
\email[Qirui Ma]{qiruimath@yeah.net}
\author[MENGYA YUE]{MENGYA YUE}
\email[Mengya Yue]{myayue@yeah.net}
\hypersetup{pdftitle={Two-element factors and finite bases for additively idempotent semirings},pdfauthor={QIRUI MA AND MENGYA YUE}}

\begin{abstract}
For every additively idempotent semiring $S$ in the constant-free signature,
$S$ is finitely based if and only if $S\times E$ is finitely based
when $E=M_2$ or $T_2$. For each $E\in\{L_2,R_2,D_2,N_2\}$,
both implications fail even for finite $S$: we give a finite finitely
based semiring whose product with $E$ is nonfinitely based and a finite
nonfinitely based semiring whose product with $E$ is finitely based.
\end{abstract}

\maketitle

\section{Introduction and main theorem}

An algebra is \emph{finitely based} if its equational theory follows
from finitely many identities. We write $\FB(S)$ for this property.
An \emph{additively idempotent semiring}, or \emph{ai-semiring}, is an
algebra $(S,+,\cdot)$ in which addition is associative, commutative and
idempotent, multiplication is associative, and multiplication distributes
over addition on both sides. Throughout the paper, the signature is
$(+,\cdot)$, with \emph{no constant symbols}; algebras are nonempty and
may be finite or infinite. General background can be found in
\cite{BS1981,HW1998}.

Up to isomorphism, there are six two-element ai-semirings
\cite{SR}. Each has addition equal to the maximum on the chain
$0<1$, and its multiplication is specified below.
\par
\begin{table}[!ht]
\centering
\caption{The six 2-element ai-semirings.}
\label{tab:two-element-semirings}
\setlength{\tabcolsep}{8pt}
\renewcommand{\arraystretch}{1.05}
\begin{tabular}{c c c c}
\hline
Semiring & $+$ & $\cdot$ & Equational basis\\
\hline
$L_2$ & $\begin{array}{@{}cc@{}}0&1\\1&1\end{array}$
      & $\begin{array}{@{}cc@{}}0&0\\1&1\end{array}$
      & $xy\approx x$\\
\hline
$R_2$ & $\begin{array}{@{}cc@{}}0&1\\1&1\end{array}$
      & $\begin{array}{@{}cc@{}}0&1\\0&1\end{array}$
      & $xy\approx y$\\
\hline
$M_2$ & $\begin{array}{@{}cc@{}}0&1\\1&1\end{array}$
      & $\begin{array}{@{}cc@{}}0&1\\1&1\end{array}$
      & $xy\approx x+y$\\
\hline
$D_2$ & $\begin{array}{@{}cc@{}}0&1\\1&1\end{array}$
      & $\begin{array}{@{}cc@{}}0&0\\0&1\end{array}$
      & $xy\approx yx,\quad x^2\approx x,\quad x+xy\approx x$\\
\hline
$T_2$ & $\begin{array}{@{}cc@{}}0&1\\1&1\end{array}$
      & $\begin{array}{@{}cc@{}}1&1\\1&1\end{array}$
      & $x_1x_2\approx y_1y_2,\quad x^2\approx x^2+x$\\
\hline
$N_2$ & $\begin{array}{@{}cc@{}}0&1\\1&1\end{array}$
      & $\begin{array}{@{}cc@{}}0&0\\0&0\end{array}$
      & $x_1x_2\approx y_1y_2,\quad x\approx x^2+x$\\
\hline
\end{tabular}
\smallskip

\begin{minipage}{\textwidth}
\footnotesize
Rows and columns are ordered $0,1$; bases are relative to the
ai-semiring axioms.
\end{minipage}
\end{table}
\par
Here $0$ and $1$ name elements, not constant symbols. Our main result
classifies the effect of these factors on finite basability.

\begin{thm}\label{intro:main}
Let $E$ be a two-element ai-semiring.
\begin{enumerate}[label=\textup{(\roman*)}]
\item If $E\in\{M_2,T_2\}$, then, for every ai-semiring $S$,
\[
 \FB(S)\quad\Longleftrightarrow\quad\FB(S\times E).
\]
\item If $E\in\{L_2,R_2,D_2,N_2\}$, there are ai-semirings
$A_E$ and $B_E$ such that
\[
 \FB(A_E),\quad \neg\FB(A_E\times E),\qquad
 \neg\FB(B_E),\quad \FB(B_E\times E).
\]
Thus neither implication holds for arbitrary $S$.
\item All witnesses in \textup{(ii)} can be chosen finite, with the
following orders:
\[
\begin{array}{c|rrrr}
E & L_2 & R_2 & D_2 & N_2\\ \hline
|A_E| & 65 & 65 & 3 & 65\\
|B_E| & 16 & 16 & 10 & 10
\end{array}.
\]
One may take $A_{N_2}=A_{L_2}$,
$A_{R_2}=A_{L_2}^{\mathrm{op}}$,
$B_{R_2}=B_{L_2}^{\mathrm{op}}$, and $B_{D_2}=B_{N_2}$.
The common ten-element witness is commutative. No minimality of
these orders is asserted.
\end{enumerate}
\end{thm}

Table~\ref{tab:classification} records the two implications separately.
Each negative entry is established by a counterexample in
Section~\ref{sec:counterexamples}.
\begin{table}[htbp]
\centering\small
\caption{The two implications for arbitrary ai-semirings $S$.}
\label{tab:classification}
\begin{tabular}{c c c}
\toprule
$E$ & $\FB(S)\Rightarrow\FB(S\times E)$ & $\FB(S\times E)\Rightarrow\FB(S)$\\
\midrule
$M_2$ & Yes & Yes\\
$T_2$ & Yes & Yes\\
$L_2$ & No & No\\
$R_2$ & No & No\\
$D_2$ & No & No\\
$N_2$ & No & No\\
\bottomrule
\end{tabular}
\end{table}

Let $\V(S)$ denote the variety generated by $S$. Since
\begin{equation}\label{intro:product-join}
 \Id(S\times E)=\Id(S)\cap\Id(E),\qquad
 \V(S\times E)=\V(S)\vee\V(E),
\end{equation}
Theorem~\ref{intro:main} is equally a classification of joins with the
six two-element varieties. The positive results are proved in
Section~\ref{sec:positive}. The $M_2$ proof gives a finite basis for
the regular identities of any finitely based ai-semiring. It also
establishes the same equivalence for adjoining a new element that absorbs
both operations, answering Gao's question \cite[Problem~6.3]{Gao2026}
affirmatively in the additively idempotent setting. Regular identities
have a general universal-algebraic theory; see \cite{Graczynska1983}.
The proof here is constructive and self-contained.

Section~\ref{sec:counterexamples} proves both negative directions
using finite witnesses. A sixty-five-element finitely based semiring
has nonfinitely based products with $L_2$ and $N_2$; its multiplicative
opposite gives the $R_2$ counterexample. Path identities show that
these products have no identity basis with a bounded number of
variables. The three-element $D_2$ example follows from
\cite{ZRCS2020,YRG}. In the reverse direction, directed quadratic
supports and odd cycles give sixteen-element witnesses for $L_2,R_2$.
A ten-element commutative semiring serves for both $N_2$ and $D_2$.
We specify the finite algebras explicitly, prove that they generate
the required varieties, and give finite bases for their products.
The path construction also disproves the conjecture of Ren and Zeng
\cite[Section~4]{RenZeng2019} concerning joins with $\V(N_2,T_2)$.
Finally, Remark~\ref{recovery:structural} gives structural classes on
which the four failed equivalences hold.

The finite basis problem for ai-semirings has a substantial literature.
Dolinka constructed a finite nonfinitely based semiring
\cite{Dolinka2007} and developed sufficient conditions for nonfinite
basability \cite{Dolinka2009INFB,Dolinka2009Remark}. Other natural
examples include the ai-semiring on the Brandt monoid
\cite{Volkov2021Brandt} and power semirings of finite solvable
non-Dedekind groups \cite{GusevVolkov2023Power}. Endomorphism semirings
have been studied for zero-preserving endomorphisms of finite
semilattices \cite{Dolinka2009End}, and for all endomorphisms of finite
chains and finite semilattices
\cite{GusevVolkov2026Chains,DolinkaGusevVolkov2025End}.

For small orders, Shao and Ren \cite{SR} determined the subvarieties
generated by the two-element ai-semirings and showed that all are
finitely based. Zhao et al.\ \cite{ZRCS2020}, together with Jackson,
Ren and Zhao \cite{JRZ}, settled the problem for order three.
Several additive-order types at order four have been classified
\cite{RLZC2025,YRZS2025,RLYC2026}; individual examples include the
finitely based semiring on $B_0$ \cite{Shaprynskii2024} and the
nonfinitely based semiring $S_7^0$ \cite{WRZ2024}. Our question
concerns the effect of a fixed direct-product factor on an arbitrary
ai-semiring, rather than a bound on its cardinality.

The theory of semilattice-ordered semigroups supplies a broader
framework \cite{KP2005}. Finite-basis classifications under additional
identities include weak commutativity for semirings with both operations
idempotent \cite{Pastijn2006} and the identity $x^3\approx x$
\cite{RZW2017}. Flat semirings provide further methods and examples:
see \cite{JRZ,RJZL} for general criteria and flat group extensions,
\cite{WZR2023} for flat nil-semirings, and \cite{ZW2025} for adjoining
a zero to a commutative flat semiring.

\section{Preliminaries}
\label{sec:prelim}

We use standard universal-algebraic terminology as in \cite{BS1981}.
Let $\AI$ denote the usual finite set of ai-semiring axioms. Variables
are written in ordinary type, and words and terms in boldface, for
example $x,y,z$ and $\mathbf{u},\mathbf{v},\mathbf{w}$. Following the
notation in \cite{YRG}, we write
\[
 \mathbf{u}\preceq\mathbf{v}
 \quad\text{for the identity}\quad
 \mathbf{u}+\mathbf{v}\approx\mathbf{v}.
\]
For elements of an ai-semiring $S$, the natural order is instead denoted
by $a\leq_S b$, meaning $a+b=b$; the subscript is omitted when the
algebra is clear. Thus $\preceq$ denotes an absorption identity,
whereas $\leq$ denotes an order relation between elements or numbers.
The adjective ``commutative'' always refers to multiplication, since
addition is already commutative.

Every term is, modulo $\AI$, a nonempty finite sum of nonempty words.
We call such an expression a polynomial and freely identify its summands
with a finite set, using additive idempotence. Its word set is denoted by
$\operatorname{supp}(\mathbf{P})$. The length of a word $\mathbf{w}$ is
$|\mathbf{w}|$. If multiplication
is also commutative, words can be treated as nonzero exponent vectors.
An identity $\mathbf{P}\approx\mathbf{Q}$ is equivalent, over $\AI$,
to the absorptions $\mathbf{w}\preceq\mathbf{Q}$ for
$\mathbf{w}\in\operatorname{supp}(\mathbf{P})$ and
$\mathbf{v}\preceq\mathbf{P}$ for
$\mathbf{v}\in\operatorname{supp}(\mathbf{Q})$.

Substitutions send variables to nonempty terms. Empty words used to denote
an absent left or right multiplicative context are only metanotation;
they never denote a constant or an admissible variable image. Additive
remainders may likewise be absent when a context is described.

For a variety $\mathcal U$, the free algebra on a countably infinite set
of generators is denoted by $F_\omega(\mathcal U)$. A variety, like an
algebra, is finitely based if its entire equational theory has a finite basis.
Since $\AI$ is finite, finite basability relative to the ai-semiring axioms
is equivalent to finite basability in the signature $(+,\cdot)$.

\begin{lem}\label{prelim:free-compactness}
Let $\mathcal U$ be a variety in a finite signature.
Then $\Id(F_\omega(\mathcal U))=\Id(\mathcal U)$.
Moreover, if $\mathcal U=\operatorname{Mod}(\Sigma)$ is finitely based,
some finite subset of $\Sigma$ already defines $\mathcal U$.
\end{lem}
\begin{proof}
An identity uses only finitely many variables. If it holds at the corresponding
canonical free generators of $F_\omega(\mathcal U)$, the universal property of the
free algebra makes it hold under every assignment in every member of
$\mathcal U$. The converse is immediate. For the second assertion,
let $\Delta$ be a finite basis for $\mathcal U$. Each member of $\Delta$
has a finite equational derivation from $\Sigma$. The union $\Sigma_0$
of the finitely many axioms used in these derivations is finite and
satisfies $\Sigma_0\vdash\Delta\vdash\Sigma$.
\end{proof}

For a word $\mathbf w$, let $h(\mathbf w)$ be its first letter.
For a polynomial $\mathbf P$, put
$h(\mathbf P)=\{h(\mathbf w):\mathbf w\in\operatorname{supp}(\mathbf P)\}$.
We will use the following general form of word-absorption closure.

\begin{lem}[Absorption closure]\label{lr:closure}
All consequences in this lemma are taken relative to $\AI$.
Rewrite a set $\Sigma$ of ai-semiring identities as absorption
identities $\mathbf{Q}\preceq \mathbf{P}$, taking both directions of each equality. Starting
from a nonempty set $A$ of words, allow the following operation:
for any substitution $\sigma$ by nonempty polynomials and any left
and right word contexts $\boldsymbol{\ell},\mathbf{r}$, add all words of
$\boldsymbol{\ell}\sigma(\mathbf{Q})\mathbf{r}$ whenever every word of $\boldsymbol{\ell}\sigma(\mathbf{P})\mathbf{r}$ is already
present. Contexts may be empty; substitution terms may not.

Let $\operatorname{cl}_{\Sigma}(A)$ be the smallest set closed under
these operations and containing $A$. Then
\[
 \Sigma\models \mathbf{w}\preceq \mathbf{P}
 \quad\Longleftrightarrow\quad
 \mathbf{w}\in\operatorname{cl}_{\Sigma}(\operatorname{supp}(\mathbf{P})).
\]
In particular, a closed set containing every word of $\mathbf{P}$ but not
$\mathbf{w}$ proves $\Sigma\not\models \mathbf{w}\preceq \mathbf{P}$.
\end{lem}
\begin{proof}
Each word addition is a valid absorption inference, which proves
the right-to-left implication. For the converse, write $c$ for the
closure operator. Closure is compatible with multiplicative
contexts. Successive applications of the closure rules give
\[
 c(c(A)\cup c(B))=c(A\cup B),\qquad
 c(c(A)c(B))=c(AB).
\]
Indeed, every finite derivation of a word from $A$ can be performed
after adjoining $B$, or after multiplying by a fixed word from
$B$; repeating this argument also handles derived words from $B$.
For closures of nonempty finite sets, define
\[
 c(A)+c(B)=c(A\cup B),\qquad
 c(A)c(B)=c(AB).
\]
The displayed compatibility equalities show that these operations
are independent of the choice of finite generating sets and define
an ai-semiring. An assignment of variables to its elements can be
represented by a choice of nonempty finite generating word sets.
These polynomial substitutions are among those permitted in
defining $c$, so this semiring satisfies every identity in
$\Sigma$.

Assign the variable $x$ to $c(\{x\})$. If
$\Sigma\models \mathbf{w}\preceq \mathbf{P}$, then
$c(\{\mathbf{w}\})\subseteq c(\operatorname{supp}(\mathbf{P}))$, and hence $\mathbf{w}$ belongs
to the latter set. Finally, each rule has finitely many premises and
conclusions, so every member of a closure is obtained after finitely
many word additions.
\end{proof}

\section{The two universal equivalences}
\label{sec:positive}

We prove both directions for $M_2$ and $T_2$ without structural or
cardinality restrictions on $S$.

\subsection{The semilattice factor \texorpdfstring{$M_2$}{M2}}

\begin{thm}\label{m:equivalence}
For every ai-semiring $S$ in the constant-free signature $(+,\cdot)$,
\[
 \FB(S\times M_2)\quad\Longleftrightarrow\quad\FB(S).
\]
Equivalently, $\V(S)\vee\V(M_2)$ is finitely based if and only if
$\V(S)$ is finitely based. No finiteness assumption on $S$ is required.
\end{thm}

For a term $\mathbf u$, write $c(\mathbf u)$ for its set of variables.
An identity is \emph{regular} if its two sides have the same content.
Thus an absorption $\mathbf p\preceq\mathbf q$ is regular precisely
when $c(\mathbf p)\subseteq c(\mathbf q)$, since it abbreviates
$\mathbf p+\mathbf q\approx\mathbf q$.

\begin{lem}\label{m:regularization}
The identities of $S\times M_2$ are exactly the regular identities of
$S$. If $S^\infty$ is obtained from $S$ by adjoining an element that
absorbs both operations, then
\[
 \V(S^\infty)=\V(S\times M_2).
\]
\end{lem}
\begin{proof}
In $M_2$, a term evaluates to the join of the values of its variables.
An identity therefore holds in $M_2$ exactly when it is regular. This
proves the first assertion. The algebra $S$ embeds in $S^\infty$, and
the partition with blocks $S$ and $\{\infty\}$ is a congruence with
quotient $M_2$. Conversely, a regular identity of $S$ holds in
$S^\infty$: an evaluation using $\infty$ gives that value on both
sides, and every other evaluation takes place in $S$. The equality of
varieties follows.
\end{proof}

\begin{proof}[Proof of Theorem~\ref{m:equivalence}]
\textbf{From $S$ to $S\times M_2$.}
Let $\Sigma$ be a finite identity basis for $S$. We construct a finite
basis for its regular identities. Include the ai-semiring axioms in
every basis considered below.

Replace each identity $\mathbf u\approx\mathbf v$ of $\Sigma$ by
its two orientations $\mathbf u\preceq\mathbf v$ and
$\mathbf v\preceq\mathbf u$. For each orientation, take the four
context forms
\[
 \mathbf u\preceq\mathbf v,\qquad
 a\mathbf u\preceq a\mathbf v,\qquad
 \mathbf u b\preceq\mathbf v b,\qquad
 a\mathbf u b\preceq a\mathbf v b,
\]
where $a,b$ are fresh variables. Denote this finite list by
$\mathcal A$. Every absorption in $\mathcal A$ is valid in $S$.

For each $\mathbf p\preceq\mathbf q$ in $\mathcal A$, put
$Y=c(\mathbf p)\setminus c(\mathbf q)$. If $Y$ is empty, retain
the absorption. Otherwise, write $Y=\{y_1,\ldots,y_k\}$ and include
all the \emph{guard identities}
\begin{equation}\label{m:guards}
 \mathbf p\preceq\mathbf q+\sum_{i=1}^{k}\mathbf g_i(y_i),
\end{equation}
where, independently for each $i$, the guard is one of
\[
 y_i,\qquad a_i y_i,\qquad y_i b_i,\qquad a_i y_i b_i.
\]
All guard parameters are pairwise distinct and absent from
$\mathbf p$ and $\mathbf q$. Include also the following two
\emph{bridge identities} for each $y\in Y$:
\begin{align}
 \mathbf p[y\mapsto uv]
 &\preceq\mathbf q+\mathbf p[y\mapsto u]+\mathbf p[y\mapsto v],
 \label{m:product-bridge}\\
 \mathbf p[y\mapsto u+v]
 &\preceq\mathbf q+\mathbf p[y\mapsto u]+\mathbf p[y\mapsto v],
 \label{m:sum-bridge}
\end{align}
where $u,v$ are fresh. Let $\Gamma$ consist of these identities and
the ai-semiring axioms.

The set $\Gamma$ is finite. Every guard identity holds in $S$ because
$\mathbf p\preceq\mathbf q$ does. The bridge identities also hold:
$y$ is absent from $\mathbf q$, so each left side is absorbed by
$\mathbf q$. All these absorptions are regular. Consequently
$\Gamma\subseteq\Id(S\times M_2)$.

We prove the following derivation property. For
$\mathbf p\preceq\mathbf q\in\mathcal A$, a substitution $\theta$
and a polynomial $\mathbf R$,
\begin{equation}\label{m:local-derivation}
 \left.
 \begin{array}{l}
  \Gamma\vdash\theta(\mathbf q)\preceq\mathbf R,\\
  c(\theta(\mathbf p))\subseteq c(\mathbf R)
 \end{array}
 \right\}
 \quad\Longrightarrow\quad
 \Gamma\vdash\theta(\mathbf p)\preceq\mathbf R.
\end{equation}
If $Y$ is empty, this follows from the retained absorption. Otherwise,
induct on the sum of the term-tree sizes of $\theta(y)$ for $y\in Y$.
When all these images are variables, select, for each
$\theta(y_i)=x_i$, a summand of $\mathbf R$ containing $x_i$.
Write it as $\boldsymbol\alpha_i x_i\boldsymbol\beta_i$, where
either context word may be empty. One of the guard forms matches this
word; substitute the nonempty context words for its parameters. The
instantiated guards are summands of $\mathbf R$, so
\eqref{m:guards} proves the desired absorption. The selected summands
may coincide.

For the induction step, let $\theta(y)$ be either
$\mathbf r\mathbf s$ or $\mathbf r+\mathbf s$. Replace this image
by $\mathbf r$ and by $\mathbf s$ to obtain substitutions
$\theta_r$ and $\theta_s$. The image of $\mathbf q$ is unchanged,
because $y\notin c(\mathbf q)$, and the images of $\mathbf p$ use
only variables already in $\theta(\mathbf p)$. The induction
hypothesis therefore gives
\[
 \Gamma\vdash\theta_r(\mathbf p)\preceq\mathbf R,
 \qquad
 \Gamma\vdash\theta_s(\mathbf p)\preceq\mathbf R.
\]
The appropriate bridge identity now yields
$\theta(\mathbf p)\preceq\mathbf R$. This proves
\eqref{m:local-derivation} for arbitrary term substitutions, including
polynomial substitutions.

Let $\mathbf P\approx\mathbf Q$ be a regular identity of $S$.
Normalize its two sides to polynomials and put
$X=c(\mathbf P)=c(\mathbf Q)$. Choose a finite equational derivation
\[
 \mathbf P=\mathbf t_0,\mathbf t_1,\ldots,\mathbf t_n=\mathbf Q
\]
from $\Sigma$, with each step replacing one occurrence of a substituted
basis term in a context. Replace all variables outside $X$ by one fixed
member of $X$. Since there are no constants, $X$ is nonempty; the
endpoints are unchanged, and every intermediate term now uses only $X$.

We show inductively that
$\Gamma\vdash\mathbf t_i\preceq\mathbf P$. A one-hole context
distributes to a polynomial of the form
\[
 \mathbf C[-]=\mathbf D+
   \sum_{j=1}^{\ell}\boldsymbol\alpha_j[-]\boldsymbol\beta_j.
\]
Here either context word may be empty and $\mathbf D$ may be absent;
this notation introduces no empty terms or constant symbols. Suppose
that the next step replaces $\mathbf C[\sigma(\mathbf v)]$ by
$\mathbf C[\sigma(\mathbf u)]$. Each term
$\boldsymbol\alpha_j\sigma(\mathbf v)\boldsymbol\beta_j$
is absorbed by $\mathbf t_i$, hence by $\mathbf P$. The corresponding
term with $\mathbf u$ in place of $\mathbf v$ has content contained
in $X$. It is an instance of one of the four context forms in
$\mathcal A$. Applying \eqref{m:local-derivation} with
$\mathbf R=\mathbf P$ shows that this new term is absorbed by
$\mathbf P$. The unchanged summands of $\mathbf D$ are already
absorbed by $\mathbf P$. Hence
$\Gamma\vdash\mathbf t_{i+1}\preceq\mathbf P$.

It follows that $\Gamma\vdash\mathbf Q\preceq\mathbf P$.
Applying the same argument to the reversed derivation, with
$\mathbf R=\mathbf Q$, gives the opposite absorption. Thus $\Gamma$
derives every regular identity of $S$ and, by
Lemma~\ref{m:regularization}, is a finite basis for $S\times M_2$.

\textbf{From $S\times M_2$ to $S$.}
Let $E=\Id(S)$ and let $E_{\mathrm r}$ be its regular part. If
$E=E_{\mathrm r}$, the assertion follows immediately from
Lemma~\ref{m:regularization}. Otherwise choose an irregular identity
$\rho:\mathbf p\approx\mathbf q$ of $S$. We show that
\begin{equation}\label{m:one-irregular}
 E=\operatorname{Eq}(E_{\mathrm r}\cup\{\rho\}).
\end{equation}
Here $\operatorname{Eq}$ denotes equational closure.
After reversing $\rho$ if necessary, choose
$y\in c(\mathbf p)\setminus c(\mathbf q)$. Substitute $z+w$ for
$y$ and $z$ for every other variable. The resulting instance has the
form $\mathbf a(z,w)\approx\mathbf b(z)$, where both $z$ and $w$
occur in $\mathbf a$. Consequently $\rho$ implies
\begin{equation}\label{m:independence}
 \mathbf a(z,w)\approx\mathbf a(z,w').
\end{equation}

Let $\mathbf s\approx\mathbf t$ be any irregular identity of $S$.
After reversing its sides if necessary, write the nonempty set
$c(\mathbf s)\setminus c(\mathbf t)$ as
$\{x_1,\ldots,x_k\}$. The following substitutions are simultaneous
for all $i$:
\begin{align*}
 \mathbf s
 &\approx\mathbf s[x_i\mapsto\mathbf a(x_i,x_i)]\\
 &\approx\mathbf s[x_i\mapsto\mathbf a(x_i,\mathbf t)]\\
 &\approx\mathbf s[x_i\mapsto\mathbf a(\mathbf t,x_i)]\\
 &\approx\mathbf s[x_i\mapsto\mathbf a(\mathbf t,\mathbf t)]
 \approx\mathbf t.
\end{align*}
The second and fourth steps follow from \eqref{m:independence}.
The first, third and last steps hold in $S$ because
$\mathbf s\approx\mathbf t$ makes the value of $\mathbf s$
independent of every $x_i$. They are regular identities: their common
contents are respectively $c(\mathbf s)$,
$c(\mathbf s)\cup c(\mathbf t)$ and $c(\mathbf t)$.
They therefore belong to $E_{\mathrm r}$, proving
\eqref{m:one-irregular}.

A finite basis for $S\times M_2$ is a finite basis for
$E_{\mathrm r}$. Adding the single identity $\rho$ consequently gives
a finite basis for $S$, completing the proof.
\end{proof}

\begin{cor}\label{m:absorber-extension}
For every ai-semiring $S$, let $S^\infty$ be the extension obtained
by adjoining a new element that absorbs addition and multiplication.
Then $S^\infty$ is finitely based if and only if $S$ is finitely based.
\end{cor}
\begin{proof}
Apply Lemma~\ref{m:regularization} and Theorem~\ref{m:equivalence}.
\end{proof}

\subsection{The constant-top factor \texorpdfstring{$T_2$}{T2}}
\label{t:subsection}

\begin{thm}\label{t:equivalence}
For every ai-semiring $S$, finite or infinite,
\[
\FB(S\times T_2)\quad\Longleftrightarrow\quad\FB(S).
\]
More explicitly, a finite basis $\Sigma$ of $S$ gives the finite basis
$\Sigma^{T}$ in \eqref{t:lift} for $S\times T_2$, relative to the
ai-semiring axioms.  Conversely, a finite basis for $S\times T_2$ can
be extended to one for $S$ by at most one additional identity.
\end{thm}

Call an ai-semiring term \emph{proper} if its distributive normal form
contains a word of length at least two.  A term that is not proper is a
nonempty sum of variables.  On $T_2$, every proper term has the constant
value $1$.  Consequently, an identity holds in $T_2$ if and only if both
sides are proper, or both sides are sums of the same set of variables.
In particular, an identity with exactly one proper side fails in $T_2$:
assigning $0$ to every variable separates its two sides.

\begin{lem}\label{t:guarded}
Let $\Sigma$ be any set of ai-semiring identities. All deductions
in this lemma are relative to the ai-semiring axioms. For each
$\mathbf p\approx\mathbf q$ in $\Sigma$, choose variables $z,t$ absent
from that identity, and let $\Sigma^{T}$ consist of the identities
\begin{equation}\label{t:lift}
\begin{split}
\mathbf p+zt&\approx\mathbf q+zt,\\
z\mathbf p&\approx z\mathbf q,\\
\mathbf pz&\approx\mathbf qz.
\end{split}
\end{equation}
If $\Sigma\vdash\mathbf u\approx\mathbf v$, then, for every word
$\mathbf w$ of length at least two,
\[
\Sigma^{T}\vdash\mathbf u+\mathbf w\approx\mathbf v+\mathbf w.
\]
If, in addition, $\mathbf u$ and $\mathbf v$ are proper, then
$\Sigma^{T}\vdash\mathbf u\approx\mathbf v$.
\end{lem}

\begin{proof}
Write $\mathbf w=\mathbf a\mathbf b$, where both factors are nonempty.
Modulo the ai-semiring axioms, an elementary step in an equational
deduction from $\Sigma$ replaces
\[
\mathbf A\,\varphi(\mathbf p)\,\mathbf B+\mathbf C
\quad\text{by}\quad
\mathbf A\,\varphi(\mathbf q)\,\mathbf B+\mathbf C,
\]
or conversely, where $\varphi$ is a substitution and
$\mathbf p\approx\mathbf q\in\Sigma$.  Either multiplicative context
$\mathbf A,\mathbf B$, and the additive context $\mathbf C$, may be
absent.  This description follows by distributing a one-hole term
context; the contexts $\mathbf A$ and $\mathbf B$ can themselves be
nonempty sums of words.  Absence of a context is only a convention and
does not introduce constants into the signature. More explicitly, the
description follows by induction on the one-hole context: addition
changes only $\mathbf C$, left multiplication changes $\mathbf A$
and $\mathbf C$, and right multiplication changes $\mathbf B$
and $\mathbf C$.

If both multiplicative contexts are absent, the first identity in
\eqref{t:lift}, with $z\mapsto\mathbf a$ and
$t\mapsto\mathbf b$, gives this step after $\mathbf w$ has been added
to both sides.  If $\mathbf A$ is present, the second identity gives
$\mathbf A\varphi(\mathbf p)\approx
\mathbf A\varphi(\mathbf q)$; multiplying on the right by
$\mathbf B$, when present, and adding the remaining context gives the
required step.  If only $\mathbf B$ is present, use the third identity.
Applying this argument to every step proves the first assertion.

For the second assertion, choose a word $\mathbf w$ of length at least
two occurring in $\mathbf u$, and such a word $\mathbf r$ occurring in
$\mathbf v$.  The first assertion and additive idempotence give
\[
\mathbf u\approx\mathbf v+\mathbf w,
\qquad
\mathbf v\approx\mathbf u+\mathbf r.
\]
The first identity implies
$\mathbf u+\mathbf v\approx\mathbf u$, and the second implies
$\mathbf u+\mathbf v\approx\mathbf v$.  Thus
$\mathbf u\approx\mathbf v$, as required.
\end{proof}

\begin{proof}[Proof of Theorem~\ref{t:equivalence}]
\textbf{From $S$ to $S\times T_2$.}
Suppose that $\Sigma$ is a finite basis for $S$.  Every identity of
$\Sigma^{T}$ holds in $S$ and has two proper sides, so it also holds
in $T_2$.  Hence $\Sigma^{T}$ holds in $S\times T_2$.
Let $\mathbf u\approx\mathbf v$ be an identity of this product.
If neither side is proper, the two sides have the same set of variables,
and the identity follows from the additive semilattice axioms.  Otherwise,
both sides are proper.  Since the identity holds in $S$, it follows from
$\Sigma$, and Lemma~\ref{t:guarded} then derives it from $\Sigma^{T}$.
This proves that $\Sigma^{T}$ is a finite basis.

\textbf{From $S\times T_2$ to $S$.}
Let $\Gamma$ be a finite basis for $S\times T_2$.  If $S$ is trivial,
then adjoining $x\approx y$ to $\Gamma$ gives a finite basis for $S$.
If $T_2\in\V(S)$,
then $S$ and $S\times T_2$ have exactly the same identities, so
$\Gamma$ is already a basis for $S$.

Suppose now that $S$ is nontrivial and $T_2\notin\V(S)$.
There is an identity of $S$ that fails in $T_2$.  A nontrivial
ai-semiring cannot satisfy an identity between two sums of variables
having different supports.  Indeed, assign an arbitrary element $a$ to
a variable occurring on only one side and an arbitrary element $b$ to
all other variables.  The identity yields either $a=b$ or $a+b=b$;
in the latter case, interchanging $a$ and $b$ also yields $a+b=a$.
Either conclusion would make the algebra trivial.  The chosen identity
therefore has exactly one proper side.  Identifying all its variables
with $x$ gives an identity of $S$ of the form
\begin{equation}\label{t:unary}
x\approx\mathbf r(x),
\end{equation}
where $\mathbf r(x)$ is a nonempty sum of positive powers of $x$ and
at least one exponent is at least two.

We claim that adjoining \eqref{t:unary} to $\Gamma$ gives a basis for $S$.
Let $\mathbf p\approx\mathbf q$ be any identity of $S$.  Both
$\mathbf r(\mathbf p)$ and $\mathbf r(\mathbf q)$ are proper terms.
The identity
$\mathbf r(\mathbf p)\approx\mathbf r(\mathbf q)$ holds in $S$
and, because both sides are proper, also in $T_2$. It therefore holds
in $S\times T_2$ and follows from
$\Gamma$.  Together with \eqref{t:unary}, this gives
\[
\mathbf p\approx\mathbf r(\mathbf p)
\approx\mathbf r(\mathbf q)\approx\mathbf q.
\]
All identities in the proposed finite basis hold in $S$, so the claim
and the theorem follow.
\end{proof}

\section{Counterexamples for the remaining four factors}
\label{sec:counterexamples}

The two implications require different witnesses, all of which will
be finite. We first construct a sixty-five-element finitely based
semiring for the $L_2$ and $N_2$ counterexamples; its multiplicative
opposite gives the $R_2$ counterexample. We then treat $D_2$ in both
directions and give the remaining reverse-direction counterexamples.

\subsection{Finite finitely based sources for \texorpdfstring{$L_2,R_2,N_2$}{L2, R2, N2}}

We construct a $65$-element finitely based ai-semiring $A$ such that
$A\times L_2$ and $A\times N_2$ are nonfinitely based. Its opposite
provides the corresponding counterexample for $R_2$. The construction
and the proofs of nonfinite basability are given explicitly below.

\subsubsection{A closure description}

Let $\mathcal W$ be the variety defined, relative to $\AI$, by
\begin{align}
 t&\preceq abc,\label{f65:top}\\
 u+tu+tv&\approx v+tu+tv,\label{f65:mark}\\
 xy&\preceq yx+x.\label{f65:reverse}
\end{align}
The variables $t,a,b,c$ in \eqref{f65:top} are independent.
Every threefold product is the same additively greatest element,
which we denote by $\Omega$. This is an element of the algebra,
not a constant symbol. It is also a multiplicative zero: for example,
if $\Omega=abc$, then $(abc)d=ab(cd)=\Omega$, and the left-sided
calculation is analogous. Consequently every polynomial containing
a word of length at least three has value $\Omega$.

A polynomial all of whose words have length at most two is described
by its singleton support $U$ and its directed quadratic support $E$.
We regard a word $xy\in E$ as an edge from $x$ to $y$, and call the
vertices in $U$ marked. Vertices in the following rules need not be
distinct.

\begin{lem}\label{f65:closure}
The exact absorption closure in $\mathcal W$ of a nonempty support
$(U,E)$ containing no word of length at least three is generated by
\begin{align}
 v\in U,\quad tu,tv\in E&\quad\Longrightarrow\quad u\in U,
 \label{f65:rule-mark}\\
 x\in U,\quad yx\in E&\quad\Longrightarrow\quad xy\in E.
 \label{f65:rule-reverse}
\end{align}
In particular, these rules add no word of length at least three,
no new unordered edge, and no new loop.
\end{lem}
\begin{proof}
Both rules are absorption consequences of the defining identities;
the two directions of \eqref{f65:mark} are covered by interchanging
$u$ and $v$. For completeness, apply Lemma~\ref{lr:closure}.
A contextual instance of the premise $abc$ in \eqref{f65:top}
always contains a word of length at least three, and hence cannot
be applied to the support under consideration. If a contextual
polynomial instance of $v+tu+tv$ contains only words of length at
most two, both multiplicative contexts are empty, and every word
in the images of $t,u,v$ has length one. For any selected letter
in the image of $u$, selecting a letter in each of the images of
$t$ and $v$ produces an instance of \eqref{f65:rule-mark}.
This accounts for every word that can be added. The same argument
for the premise $yx+x$ gives \eqref{f65:rule-reverse}.
Thus arbitrary polynomial substitutions yield no additional rules.
The last assertions follow directly from the two displayed rules.
\end{proof}

Let $(U,E)$ now be closed under these rules. For a vertex $v$, set
\[
 m(v)=\begin{cases}1&v\in U,\\0&v\notin U,\end{cases}
 \qquad
 a(v)=\begin{cases}1&vw\in E\text{ for some }w\in U,\\
                    0&\text{otherwise}.\end{cases}
\]
Every edge $vw\in E$ satisfies
\begin{equation}\label{f65:types-equations}
 m(w)=a(v),\qquad m(w)=1\ \Longrightarrow\ a(w)=m(v).
\end{equation}
Indeed, a marked successor of $v$ forces all its successors to be
marked by \eqref{f65:rule-mark}. If $w$ is marked,
\eqref{f65:rule-reverse} also gives $wv\in E$, to which the first
equality can be applied. Denote the four pairs $(m(v),a(v))$ by
\[
 F=(1,1),\qquad A=(1,0),\qquad B=(0,1),\qquad C=(0,0).
\]
It follows from \eqref{f65:types-equations} that every edge has a
type pair in
\begin{equation}\label{f65:allowed}
 \mathcal D=\{FF,AB,AC,BA,CB,CC\}.
\end{equation}
The marked types are precisely $F$ and $A$.

\subsubsection{Six finite syntactic quotients}

For each $IJ\in\mathcal D$, use the alphabet
$X=\{f,a,b,c,p,q\}$, where $f,a,b,c$ have types $F,A,B,C$,
respectively, and $p,q$ have types $I,J$, respectively. Let
$K_{IJ}\subseteq X^+$ consist of all singleton letters of marked
type and all quadratic words whose type pair belongs to
$\mathcal D$, except that $pq$ is omitted; if $I\in\{F,A\}$,
omit $qp$ as well. No longer word belongs to $K_{IJ}$.

Each $K_{IJ}$ is closed under \eqref{f65:rule-mark} and
\eqref{f65:rule-reverse}. To verify the first rule, a vertex with
an allowed marked successor has type $F$ or $B$, and all its
allowed successors are marked. Removing edges cannot invalidate
this implication. For the second rule, an allowed edge with marked
target has type $FF$ or $BA$, and its reverse is allowed. A reversal
could produce the deleted edge $pq$ only when $p$ is marked and
$qp$ is present; our additional deletion excludes exactly this
possibility. A reversal producing a deleted $qp$ would require the
already deleted edge $pq$.

Let $\mathcal P_f^+(X^+)$ be the semiring of nonempty finite
languages over $X$, under union and concatenation. For a language
$K\subseteq X^+$, define
\begin{equation}\label{f65:syntactic}
 P\equiv_K Q
 \quad\Longleftrightarrow\quad
 \bigl(\ell P r\subseteq K\ \Longleftrightarrow\ \ell Qr\subseteq K\bigr)
 \quad\text{for every }\ell,r\in X^*.
\end{equation}
Empty contexts are allowed in this definition; they do not add
constants to the signature. This is a semiring congruence. Indeed,
a subset test for a union is the conjunction of the corresponding
tests, and a subset test after multiplication by a finite language
is the conjunction of the tests for its individual word contexts.
Put
\[
 Q_{IJ}=\mathcal P_f^+(X^+)/{\equiv_{K_{IJ}}}.
\]

For clarity, the sizes of these quotients can be checked without
constructing any multiplication tables. The signature of $P$ is
the $13$-tuple of truth values of the subset tests with contexts
\begin{equation}\label{f65:contexts}
 (\varepsilon,\varepsilon),\quad
 (f,\varepsilon),(a,\varepsilon),(b,\varepsilon),
 (c,\varepsilon),(p,\varepsilon),(q,\varepsilon),\quad
 (\varepsilon,f),(\varepsilon,a),(\varepsilon,b),
 (\varepsilon,c),(\varepsilon,p),(\varepsilon,q).
\end{equation}
All other tests are false. Union corresponds to coordinatewise
intersection of signatures. A language containing a word of length at least two
has all twelve nonempty-context coordinates equal to zero; its
signature is therefore either the signature of $f+a$ or that of
$f+b$. These are, respectively, $(1,0,\ldots,0)$ and
$(0,0,\ldots,0)$. Every remaining language is a nonempty subset of
$X$. Table~\ref{f65:quotient-table} lists one representative of
each possible signature; addition in the table means union of
singleton words. For each row, the listed signatures are distinct,
contain the six singleton signatures, and are closed under
intersection with each of those six signatures. These assertions
follow by the thirteen tests in \eqref{f65:contexts}, using the
explicit definition of $K_{IJ}$. Thus the table is an exhaustive
finite specification of the quotient elements, not just an upper
bound for their number.

\begin{table}[ht]
\centering
\small
\setlength{\tabcolsep}{5pt}
\caption{Representatives of the elements of the six syntactic quotients.}
\label{f65:quotient-table}
\begin{tabular}{@{}clc@{}}
\toprule
$IJ$ & Representatives & $|Q_{IJ}|$\\
\midrule
$FF$ & $f,a,b,c,p,q,\ p+q,\ a+c,\ b+c,\ f+a,\ f+b$ & $11$\\
$AB$ & $f,a,b,c,p,q,\ c+p,\ a+c,\ c+q,\ b+c,\ f+a,\ f+b$ & $12$\\
$AC$ & $f,a,b,c,p,q,\ c+p,\ a+c,\ b+q,\ b+c,\ f+a,\ f+b$ & $12$\\
$BA$ & $f,a,b,c,p,q,\ a+c,\ b+c,\ f+a,\ f+b$ & $10$\\
$CB$ & $f,a,b,c,p,q,\ a+p,\ a+c,\ c+q,\ b+c,\ f+a,\ f+b$ & $12$\\
$CC$ & $f,a,b,c,p,q,\ a+p,\ a+c,\ b+q,\ p+q,\ b+c,\ f+a,\ f+b$ & $13$\\
\bottomrule
\end{tabular}
\end{table}

\begin{lem}\label{f65:generate}
The six quotients belong to $\mathcal W$, and
\[
 \mathcal W=\V(Q_{FF},Q_{AB},Q_{AC},Q_{BA},Q_{CB},Q_{CC}).
\]
\end{lem}
\begin{proof}
We first check the defining identities in every quotient. For an
absorption form of \eqref{f65:mark} or \eqref{f65:reverse}, consider
any polynomial substitution and any word contexts. If the premise
is not contained in $K_{IJ}$, adding the conclusion cannot change
that failed subset test. If it is contained in $K_{IJ}$, the
length argument in the proof of Lemma~\ref{f65:closure} reduces
the inference to the two word rules, under which $K_{IJ}$ is closed.
Thus both sides have the same signature. Every threefold product
has the all-zero signature, represented by any word of length
three. This element is additively absorbing and a multiplicative
zero, so \eqref{f65:top} holds as well.

For the reverse inclusion, it suffices to separate every absorption
$\mathbf w\preceq\mathbf P$ that fails in $\mathcal W$. Its premise
contains no word of length at least three. Replace its support by
its closure $(U,E)$ from Lemma~\ref{f65:closure}; this uses only the
finitely many vertices occurring in the premise. Include any
additional variables of $\mathbf w$ as isolated unmarked vertices,
of type $C$. Assign to each vertex its type as above.

If $|\mathbf w|\geqslant3$, send every variable to $f$ in any one
of the quotients. The premise consists of $f$ and $ff$, both in
$K_{IJ}$, whereas the target word is not in $K_{IJ}$. The empty-context
test separates the two sides. If $\mathbf w=x$ is a singleton,
then $x\notin U$. Send each vertex to the generic letter
$f,a,b,c$ for its type. All premise words belong to $K_{IJ}$,
whereas the image of $x$ has type $B$ or $C$ and is not a singleton
in $K_{IJ}$.

Suppose next that $\mathbf w=xy$ with $x\neq y$. If its type
pair is not in $\mathcal D$, the same generic assignment separates
it. Otherwise, use the quotient $Q_{IJ}$ for its type pair,
send $x$ to $p$ and $y$ to $q$, and send all other vertices to
their generic letters. No premise word maps to $pq$, since the
images $p,q$ are unique to $x,y$ and $xy\notin E$. If $p$ is
marked, then $yx\notin E$ as well: otherwise
\eqref{f65:rule-reverse} would add $xy$. Hence no premise word
maps to the additional deleted word $qp$. The premise is contained
in $K_{IJ}$, but the target maps to $pq\notin K_{IJ}$.

Finally, suppose that $\mathbf w=x^2$. If $x$ has type $A$ or
$B$, its type pair is forbidden and the generic assignment suffices.
If it has type $F$ or $C$, use $Q_{FF}$ or $Q_{CC}$, respectively,
and substitute $p+q$ for $x$, sending all other vertices to their
generic letters. Since $x^2\notin E$, every premise quadratic
contains $x$ at most once. Its expansion therefore contains neither
$pq$ nor $qp$, and all its expanded words have allowed type pairs.
Every singleton premise still has marked type. Thus the premise
remains inside the language, whereas $(p+q)^2$ contains the excluded
word $pq$. This proves separation in every case.
\end{proof}

Take disjoint copies of the six $Q_{IJ}$ and identify their
all-zero-signature elements $\Omega$. Retain the operations within
each component, and set both the sum and product of elements from
different components equal to $\Omega$. Denote the resulting
algebra by $A$. It embeds into the direct product of the six
quotients by putting an element in its own coordinate and $\Omega$
in the other coordinates. It also contains each quotient as a
subalgebra. Therefore Lemma~\ref{f65:generate} gives
\begin{equation}\label{f65:source}
 |A|=11+12+12+10+12+13-5=65,
 \qquad \V(A)=\mathcal W.
\end{equation}
In particular, $A$ is an ai-semiring with the finite identity basis
$\AI$ and \eqref{f65:top}--\eqref{f65:reverse}.

\subsubsection{The null factor}
\label{n:forward-section}

\begin{thm}\label{nf:forward-counterexample}
The $65$-element ai-semiring $A$ in \eqref{f65:source} is finitely
based, whereas $A\times N_2$ is nonfinitely based. In fact, the
identities of $A\times N_2$ have no basis with a bounded number
of variables.
\end{thm}
\begin{proof}
For $n\geqslant1$, use distinct variables
$z_0,\ldots,z_n,t_1,\ldots,t_n,r$, and put
\begin{equation}\label{f65:null-path}
 \mathbf P_n=z_0+\sum_{i=1}^{n}
 \bigl(z_{i-1}t_i+t_i z_{i-1}+t_i z_i+z_i t_i\bigr)+rz_n.
\end{equation}
Its quadratic graph consists of both orientations of the path
$z_0,t_1,z_1,\ldots,t_n,z_n$, together with the edge $r\longrightarrow z_n$.
Rule \eqref{f65:rule-mark} successively marks $z_1,\ldots,z_n$,
and \eqref{f65:rule-reverse} then adds $z_nr$. The exact closure
is consequently
\begin{equation}\label{f65:null-closure}
 \operatorname{supp}(\mathbf P_n)\cup
 \{z_1,\ldots,z_n,z_nr\}.
\end{equation}
Indeed, this set is closed: every edge joins a vertex in
$\{z_0,\ldots,z_n\}$ to one in $\{t_1,\ldots,t_n,r\}$, all
vertices in the first part are marked, and all edges are now
bidirected. Hence
\begin{equation}\label{f65:null-identity}
 z_nr\preceq\mathbf P_n
\end{equation}
holds in $\mathcal W$. It also holds in $N_2$, where every word of
length at least two has value $0$. Thus it holds in $A\times N_2$.

Suppose that these identities have a basis using at most $k$
variables in each identity. Replace its equations by single-word
absorptions in both directions; this does not increase the variable
bound. Choose $n>k$. By Lemma~\ref{lr:closure}, there is an
increasing derivation of \eqref{f65:null-identity} from
$\operatorname{supp}(\mathbf P_n)$. Every added word belongs to
\eqref{f65:null-closure}, because each inference is valid in
$\mathcal W$. An identity of $N_2$ preserves exactly the set of
singleton words: products evaluate to $0$, and distinct singleton
supports are separated by assigning one variable the value $1$
and all others $0$. Consequently no new singleton can be added
in this derivation. Its first new word must therefore be $z_nr$.

Write the basis absorption used at that step as
$\mathbf u\preceq\mathbf Q$, with substitution $\sigma$ and
word contexts $\boldsymbol\ell,\mathbf s$. The source target
$\mathbf u$ cannot be a variable. If it were a variable $v$,
validity in $N_2$ would require the singleton $v$ to occur in
$\mathbf Q$, so all words of
$\boldsymbol\ell\sigma(v)\mathbf s$ would already be present.
Since the new word has length two, it follows that both contexts
are empty and $|\mathbf u|=2$. Nor can $\mathbf u$ be a square:
if $\sigma(v)^2$ contains $z_nr$, then $\sigma(v)$ contains
both singleton letters $z_n,r$, so its square also contains the
loop $z_n^2$, contrary to \eqref{f65:null-closure}. Thus
$\mathbf u=uv$ for distinct source variables $u,v$.

Choose one word from the image under $\sigma$ of every source
variable, choosing $z_n$ for $u$ and $r$ for $v$. This defines a
word substitution $\theta$. All words in
$\mathbf D=\theta(\mathbf Q)$ belong to
$\operatorname{supp}(\mathbf P_n)$, and validity of the source
absorption in $\mathcal W$ gives
$\mathcal W\models z_nr\preceq\mathbf D$.
By Lemma~\ref{f65:closure}, producing the absent edge $z_nr$
requires that $z_n$ become marked. Neither rule changes the
underlying undirected graph, and a new mark always lies in the
connected component of an existing mark. The only possible
initial mark in $\mathbf D$ is $z_0$. Hence $z_0$ must occur
and its component in the graph of $\mathbf D$ must contain $z_n$.
As this graph is a subgraph of the displayed path with the
additional leaf $r$, it must contain all $2n+1$ vertices
$z_0,t_1,z_1,\ldots,t_n,z_n$.

On the other hand, every source variable occurring in $\mathbf Q$
has a selected image of length at most two: that image is a factor
of a word of $\mathbf D\subseteq\operatorname{supp}(\mathbf P_n)$.
The at most $k$ source variables therefore supply at most $2k$
distinct letters. This contradicts $2n+1>2k$ and proves the theorem.
\end{proof}

Theorem~\ref{nf:forward-counterexample} also gives a finite-algebra
counterexample to the conjecture of Ren and Zeng
\cite[Section~4]{RenZeng2019} that joining a finitely based
ai-semiring variety with $\V(N_2,T_2)$ preserves finite basability.
Indeed, $T_2$ satisfies \eqref{f65:top}--\eqref{f65:reverse}, so
\[
 \mathcal W\vee\V(N_2,T_2)=\mathcal W\vee\V(N_2)
 =\V(A\times N_2).
\]
Here $\mathcal W$ is finitely based and generated by a finite
algebra, whereas the displayed join is nonfinitely based.

\subsubsection{The left-zero and right-zero factors}
\label{lf:forward-section}

\begin{thm}\label{lf:forward-counterexample}
The ai-semirings $A$ and $A^{\mathrm{op}}$ are finitely based and
have $65$ elements. Their products satisfy
\[
 \neg\FB(A\times L_2),\qquad
 \neg\FB(A^{\mathrm{op}}\times R_2).
\]
Neither product has an identity basis with a bounded number of
variables.
\end{thm}
\begin{proof}
For distinct variables $z_0,\ldots,z_n,t_1,\ldots,t_n,g$, put
\begin{equation}\label{f65:left-path}
 \mathbf H_n=z_0+\sum_{i=1}^{n}(t_i z_{i-1}+t_i z_i)+z_ng.
\end{equation}
Successive applications of \eqref{f65:rule-mark} show that
$\mathcal W\models z_n\preceq\mathbf H_n$. This absorption also
holds in $L_2$, since $z_n$ occurs as the head of $z_ng$.
Thus it holds in $A\times L_2$.

The exact $\mathcal W$-closure of the support of $\mathbf H_n$
marks precisely $z_0,\ldots,z_n$ and adds the reverse edges
$z_{i-1}t_i,z_it_i$ for $1\leqslant i\leqslant n$.
The edge $z_ng$ retains its original orientation. To verify
closedness, note that every edge crosses the bipartition
$\{z_0,\ldots,z_n\}$ and $\{t_1,\ldots,t_n,g\}$, and all
vertices in the first part are marked; the only extra leaf edge
has unmarked target and cannot be reversed by
\eqref{f65:rule-reverse}. In particular, no loop or word of
length at least three is introduced, and the underlying graph
is the path $z_0,t_1,z_1,\ldots,t_n,z_n$ with the additional
leaf $g$ at $z_n$.

Suppose that the product has an identity basis with at most $k$
variables per identity, split into single-word absorptions as
before, and take $n>k$. In an increasing derivation of
$z_n\preceq\mathbf H_n$, all added words lie in the exact
$\mathcal W$-closure just described. Moreover, validity in $L_2$
preserves the set of heads, initially
\[
 h(\mathbf H_n)=\{z_0,t_1,\ldots,t_n,z_n\}.
\]
Thus the only singleton that can be newly added is $z_n$:
the intermediate singletons $z_1,\ldots,z_{n-1}$ are excluded
by their heads, and the other vertices are never marked in the
$\mathcal W$-closure.

Consider the first step adding $z_n$. Because this target has
length one, the multiplicative contexts are empty and the source
target is a variable $v$. Choose the image $z_n$ for $v$ and
one word from each other source-variable image. If the source
absorption is $v\preceq\mathbf Q$, denote the resulting word
instance of its premise by $\mathbf D$. Its only possible
singleton is $z_0$, its edges lie in the graph just described,
and $\mathcal W\models z_n\preceq\mathbf D$. By
Lemma~\ref{f65:closure}, deriving the mark $z_n$ therefore
requires $z_0$ and an undirected connection from $z_0$ to $z_n$
in the graph of $\mathbf D$. The unique path in the ambient
tree contains all $2n+1$ vertices
$z_0,t_1,z_1,\ldots,t_n,z_n$. Every source variable occurring in
$\mathbf Q$ has a selected image of length at most two, since that
image is a factor of a word of $\mathbf D$. Hence at most $2k$ distinct
letters can occur in $\mathbf D$, a contradiction.

This proves the assertion for $A\times L_2$. Reversing every
multiplicative word preserves equational consequence and the
number of variables, and gives
\[
 (A\times L_2)^{\mathrm{op}}\cong A^{\mathrm{op}}\times R_2.
\]
It also sends a finite basis for $A$ to a finite basis for
$A^{\mathrm{op}}$. The remaining assertions follow.
\end{proof}

\subsection{The lattice factor \texorpdfstring{$D_2$}{D2}: both directions}
\label{d:section}

\subsubsection{A finite counterexample to the forward implication}
\label{subsec:d2-forward}

The forward implication already fails for a three-element semiring.
Let $S_{53}=\{1,a,\infty\}$, with addition equal to the maximum in
$1<a<\infty$. Multiplication has identity $1$, absorbing element
$\infty$, and $a^2=\infty$. Thus its operation tables are
\[
\begin{array}{c|ccc}
+&1&a&\infty\\\hline
1&1&a&\infty\\
a&a&a&\infty\\
\infty&\infty&\infty&\infty
\end{array}
\qquad
\begin{array}{c|ccc}
\cdot&1&a&\infty\\\hline
1&1&a&\infty\\
a&a&\infty&\infty\\
\infty&\infty&\infty&\infty
\end{array}.
\]
The symbols $1$ and $\infty$ name elements; they are not constants in
the algebraic signature. This is the semiring denoted by $S_{53}$ in
\cite{ZRCS2020,YRG}.

\begin{thm}\label{d:forward-counterexample}
The semiring $S_{53}$ is finitely based, whereas $S_{53}\times D_2$
is nonfinitely based.
\end{thm}
\begin{proof}
By \cite[Proposition~7]{ZRCS2020}, the following four identities,
together with the ai-semiring axioms, form a basis for $S_{53}$:
\begin{align}
 xy&\approx yx,\label{d:s53-comm}\\
 xy+y^2&\approx x+y^2,\label{d:s53-square}\\
 x+xy&\approx xy,\label{d:s53-absorb}\\
 xy+yz+xz&\approx xyz.\label{d:s53-triple}
\end{align}
Consequently $S_{53}$ is finitely based.

Yue, Ren and Gao prove in \cite[Proposition~3.4]{YRG} that
\[
 \V(S_{(4,545)})=\V(S_{53})\vee\V(D_2).
\]
They also prove that $S_{(4,545)}$ is nonfinitely based
\cite[Theorem~3.10]{YRG}. Since the identities of a direct product
are exactly the identities common to its factors,
\[
 \V(S_{53}\times D_2)
 =\V(S_{53})\vee\V(D_2)
 =\V(S_{(4,545)}).
\]
The algebras $S_{53}\times D_2$ and $S_{(4,545)}$ therefore have
identical equational theories. The second is nonfinitely based, so the
first is nonfinitely based as well.
\end{proof}

In particular, this failure occurs for a finite $S$ whose additive
reduct is a chain. A chain assumption on $S$ alone cannot make the
$D_2$ equivalence valid.

\subsubsection{A ten-element counterexample to the reverse implication}

The reverse implication fails for the same ten-element semiring that
will be used for the null factor. We use its explicit construction and
identity theory from Subsection~\ref{n:section}.

\begin{thm}\label{d:reversecounter}
Let $C_{10}=Q\oplus_0 B$ be the ten-element commutative ai-semiring
in \eqref{n:S-definition}. Then
\[
 \neg\FB(C_{10}),\qquad \FB(C_{10}\times D_2).
\]
A finite identity basis for $C_{10}\times D_2$, relative to $\AI$,
is given by
\begin{align}
 xy&\approx yx,\label{d:finite-comm}\\
 x+x^2&\approx x^2,\label{d:finite-square}\\
 abc+x^2&\approx abc+x,\label{d:finite-contract}\\
 abc+tx^2&\approx abc+tx,\label{d:finite-contract-context}\\
 abc+x+xy&\approx abc+x.\label{d:finite-absorb}
\end{align}
Every product of three elements of $C_{10}$ is its additive greatest
element and multiplicative zero.
\end{thm}

\begin{proof}
By Lemmas~\ref{n:finite-variety-nfb} and \ref{n:finite-generator},
$C_{10}$ generates the nonfinitely based variety $\mathcal U$
defined in Subsection~\ref{n:section}. In particular, $C_{10}$ is
nonfinitely based. We prove that the displayed identities form a
basis for its product with $D_2$.

Let $\Gamma$ consist of $\AI$ and
\eqref{d:finite-comm}--\eqref{d:finite-absorb}.
All its identities hold in $D_2$. They also hold in $\mathcal U$:
commutativity is a defining identity, and
\eqref{d:finite-square} follows from $\varepsilon_3$ by identifying
all its variables with $x$. Each of the remaining identities has the
cubic summand $abc$ on both sides, so both sides are the additive
greatest element in every member of $\mathcal U$. Thus $\Gamma$
is valid in $C_{10}\times D_2$.

We first record a consequence of the last three identities. Fix a word
$\mathbf r$ of length at least three. For a polynomial $\mathbf f$,
let $\operatorname{nf}_D(\mathbf f)$ be its squarefree polynomial
whose summands have the inclusion-minimal variable sets among the
summands of $\mathbf f$. This is the usual disjunctive normal form
for $D_2$: two polynomials agree in $D_2$ if and only if these
antichains of variable sets coincide. Indeed, a squarefree polynomial
takes value $1$ precisely when the set of variables assigned $1$
contains the variable set of one of its summands. Its
inclusion-minimal such sets are exactly that antichain.

The identities in $\Gamma$ prove
\begin{equation}\label{d:finite-guarded-normal-form}
 \mathbf r+\mathbf f
 \approx \mathbf r+\operatorname{nf}_D(\mathbf f).
\end{equation}
To see this, split $\mathbf r$ into three nonempty factors and
substitute them for $a,b,c$ in
\eqref{d:finite-contract}--\eqref{d:finite-absorb}. Retain
$\mathbf r$ as an additional summand throughout the following
reductions. By commutativity, a repeated letter can be written as
$x^2$ with its remaining factors collected together. If there is no
remaining factor, apply \eqref{d:finite-contract}; otherwise, apply
\eqref{d:finite-contract-context}. Repeated application makes every
summand squarefree. If squarefree summands $\mathbf u,\mathbf v$
have strictly nested variable sets, with that of $\mathbf u$
contained in that of $\mathbf v$, write
$\mathbf v=\mathbf u\mathbf h$, where $\mathbf h$ is nonempty.
Identity~\eqref{d:finite-absorb}, with $x=\mathbf u$ and
$y=\mathbf h$, deletes $\mathbf v$. Additive idempotence deletes
duplicates. These reductions prove
\eqref{d:finite-guarded-normal-form}. Consequently,
\begin{equation}\label{d:finite-guarded-lifting}
 D_2\models\mathbf f\approx\mathbf g
 \quad\Longrightarrow\quad
 \Gamma\vdash\mathbf r+\mathbf f\approx\mathbf r+\mathbf g.
\end{equation}

Now let $\mathbf f\approx\mathbf g$ be any identity of
$C_{10}\times D_2$. It holds in both $\mathcal U$ and $D_2$.
By Lemma~\ref{n:finite-normal-forms}, either both polynomials
contain a word of length at least three or neither does.

Suppose first that both contain such a word. Choose one,
$\mathbf r$, from $\mathbf f$ and one, $\mathbf s$, from
$\mathbf g$. Since the two polynomials agree in $D_2$, so do
$\mathbf f$ and $\mathbf f+\mathbf g$, and likewise
$\mathbf g$ and $\mathbf f+\mathbf g$. Applying
\eqref{d:finite-guarded-lifting} with $\mathbf r$ and then with
$\mathbf s$, and using their occurrence among the corresponding
summands, gives
\[
 \Gamma\vdash\mathbf f\approx\mathbf f+\mathbf g,
 \qquad
 \Gamma\vdash\mathbf g\approx\mathbf f+\mathbf g.
\]
Hence $\Gamma\vdash\mathbf f\approx\mathbf g$.

Suppose instead that all summands have length at most two. Write
$(A,E)$ and $(A',E')$ for their singleton and commutative quadratic
supports. By Lemma~\ref{n:finite-normal-forms}, equality in
$\mathcal U$ implies $E=E'$. Let
\[
 L(E)=\{x:x^2\in E\}.
\]
Assigning $x=1$ and every other variable $0$ in $D_2$ shows that
$\mathbf f$ takes value $1$ exactly when $x\in A\cup L(E)$;
the corresponding condition for $\mathbf g$ is
$x\in A'\cup L(E)$. Therefore
\[
 A\cup L(E)=A'\cup L(E).
\]
Every singleton on which $A$ and $A'$ differ is thus accompanied
by its square on both sides. Identity~\eqref{d:finite-square}
inserts or removes that singleton without changing the square.
Applying it to all such differences, together with the commutative
ai-semiring laws, proves $\mathbf f\approx\mathbf g$.
This proves completeness of $\Gamma$ and finishes the proof.
\end{proof}

\subsection{The left-zero and right-zero factors: the reverse implication}
\label{lr:section}

We construct a sixteen-element nonfinitely based ai-semiring whose
product with $L_2$ is finitely based. Its multiplicative opposite
provides the corresponding witness for $R_2$. Recall that
$h(\mathbf{w})$ denotes the first letter of a word and
$h(\mathbf{P})$ the set of first letters of a polynomial.
Since multiplication in
$L_2$ is left-zero multiplication,
\begin{equation}
 L_2\models \mathbf{w}\preceq \mathbf{P}\quad\Longleftrightarrow\quad h(\mathbf{w})\in h(\mathbf{P}).
 \label{lr:leftcriterion}
\end{equation}
Indeed, a word evaluates in $L_2$ to the value of its first letter.

\begin{thm}\label{lr:counterexample}
There exists a sixteen-element ai-semiring $S$ such
that all products of three elements have the same value, which is both
an additive greatest element and a multiplicative zero, and
\[
 \neg\FB(S),\qquad \FB(S\times L_2).
\]
Its multiplicative opposite $S^{\mathrm{op}}$ satisfies
\[
 \neg\FB(S^{\mathrm{op}}),\qquad
 \FB(S^{\mathrm{op}}\times R_2).
\]
Consequently, neither of the implications
$\FB(S\times L_2)\Rightarrow\FB(S)$ and
$\FB(S\times R_2)\Rightarrow\FB(S)$ holds for arbitrary ai-semirings.
\end{thm}

We begin with a finitely based variety containing the
nonfinitely based variety used in the construction. Set
\begin{equation}
 \mathcal C=\operatorname{Mod}\bigl(\AI,\ xy\approx yx,\ t\preceq abc\bigr).
 \label{lr:Cdefinition}
\end{equation}
Here and below, variables are distinct unless explicitly
identified. The identity $t\preceq abc$
forces every product of three elements to be the additive greatest
element, which is necessarily unique. By associativity, every product
of four elements can also be written as a product of three elements.
Thus this greatest element is a multiplicative zero. In particular,
all words of length at least three have the same value.

\begin{lem}\label{lr:Ccriterion}
Let $\mathbf{w}$ be a word and $\mathbf{P}$ a nonempty polynomial. Then
$\mathcal C\models \mathbf{w}\preceq \mathbf{P}$ if and only if one of the following conditions
holds:
\begin{enumerate}
\item $\mathbf{P}$ contains a word of length at least three;
\item every word in $\mathbf{P}$ has length at most two, and the commutative
image of $\mathbf{w}$ equals the commutative image of a word in $\mathbf{P}$.
\end{enumerate}
\end{lem}
\begin{proof}
The defining identities establish sufficiency. For necessity,
let $X$ be a set of variables containing those under consideration,
and let $D$ be the set of nonempty commutative words over $X$ of
degree at most two. Consider the algebra whose elements are the
nonempty finite subsets of $D$, together with an additional element
$\Omega$. Addition of subsets is union. Their product is the
setwise product, unless a product word has degree at least three,
in which case the result is $\Omega$. The element $\Omega$ absorbs
both operations.

Associativity and distributivity follow from the corresponding
properties of union and setwise product and from the fact that
multiplication by a nonempty word cannot lower the degree of a word
of degree at least three. The algebra is
commutative and every threefold product is $\Omega$, so it belongs
to $\mathcal C$. Under the assignment $x\mapsto\{x\}$, a polynomial
without a word of length at least three evaluates to its set of
commutative words, whereas a polynomial containing such a word
evaluates to $\Omega$. Therefore, if the first condition fails and
$\mathbf{w}\preceq \mathbf{P}$ holds, the commutative image of $\mathbf{w}$ must occur in $\mathbf{P}$.
\end{proof}

\begin{lem}\label{lr:Wbasis}
The variety $\mathcal W=\mathcal C\vee\V(L_2)$ is defined by the
ai-semiring axioms and the following five identities:
\begin{align}
 xyz&\approx xzy, \tag{L1}\label{lr:L1}\\
 xy&\preceq yx+xz, \tag{L2}\label{lr:L2}\\
 xy&\preceq yx+x, \tag{L3}\label{lr:L3}\\
 x&\preceq abc+xy, \tag{L4}\label{lr:L4}\\
 xy&\preceq abc+x. \tag{L5}\label{lr:L5}
\end{align}
\end{lem}
\begin{proof}
Each of the five identities holds in both $\mathcal C$ and $L_2$. For
completeness, it suffices to derive every absorption identity
$\mathbf{w}\preceq \mathbf{P}$ valid in both, since every polynomial identity is equivalent
to a finite set of such identities.

Write $x=h(\mathbf{w})$. By \eqref{lr:leftcriterion}, there is a word
$\mathbf{u}$ in $\mathbf{P}$ with $h(\mathbf{u})=x$.

First suppose that $\mathbf{P}$ contains a word $\mathbf{q}$ of length at least three.
Factor $\mathbf{q}$ into three nonempty words. Substitution in
\eqref{lr:L4} and \eqref{lr:L5} gives
\begin{equation}
 \mathbf{A}+\mathbf{q}\approx \mathbf{A}\mathbf{B}+\mathbf{q}
 \label{lr:modq}
\end{equation}
for arbitrary nonempty terms $\mathbf{A},\mathbf{B}$. Since both $\mathbf{w}$ and $\mathbf{u}$ begin
with $x$, this yields
\[
 \mathbf{w}+\mathbf{q}\approx x+\mathbf{q}\approx \mathbf{u}+\mathbf{q}.
\]
The same equality is immediate when either word is the
one-letter word $x$. Hence $\mathbf{w}\preceq \mathbf{u}+\mathbf{q}\preceq \mathbf{P}$.

Now suppose that $\mathbf{P}$ has no word of length at least three.
Lemma~\ref{lr:Ccriterion} provides a word $\mathbf{v}$ in $\mathbf{P}$
with the same commutative image as $\mathbf{w}$; in particular,
$|\mathbf{w}|\leq2$. If $\mathbf{w}=\mathbf{v}$, the desired
absorption is immediate. Otherwise $\mathbf{w}=xy$ and
$\mathbf{v}=yx$ for distinct variables $x,y$. The head witness
$\mathbf{u}$ is either $x$ or $xz$ for some variable $z$.
Identity~\eqref{lr:L3} gives $xy\preceq yx+x$ in the first case,
and \eqref{lr:L2} gives $xy\preceq yx+xz$ in the second.
Thus $\mathbf{w}\preceq\mathbf{v}+\mathbf{u}\preceq\mathbf{P}$
in both cases.
\end{proof}

For each odd integer $n\geq3$, put
\begin{equation}
 \varepsilon_n:\qquad
 xy\preceq yx+z_1z_2+z_2z_3+\cdots+z_nz_1.
 \label{lr:epsilon}
\end{equation}
Define
\begin{equation}
 \mathcal V=\operatorname{Mod}\bigl(
 \AI,\ \text{\eqref{lr:L1}--\eqref{lr:L3}},\ t\preceq abc,\quad
 \varepsilon_n\ (n\geq3\text{ odd})\bigr).
 \label{lr:Vdefinition}
\end{equation}
Each $\varepsilon_n$ holds in every commutative ai-semiring, so
$\mathcal C\subseteq\mathcal V$. Conversely, $t\preceq abc$ implies
\eqref{lr:L4} and \eqref{lr:L5}; thus
$\mathcal V\subseteq\mathcal W$. Lemma~\ref{lr:Wbasis} now gives
\begin{equation}
 \mathcal V\vee\V(L_2)=\mathcal W.
 \label{lr:join}
\end{equation}

\begin{lem}\label{lr:Vnfb}
The variety $\mathcal V$ is not finitely based.
\end{lem}
\begin{proof}
Fix $N\geq3$. Let $\Delta_N$ comprise the finite set of background
identities in \eqref{lr:Vdefinition}, together with all
$\varepsilon_n$ for odd $n$ satisfying $3\leq n\leq N$. Choose an
odd integer $m>N$ and distinct variables
\[
 x,y,c_0,c_1,\ldots,c_{m-1},
\]
where the subscripts on the $c_i$ are read modulo $m$. Set
\[
 \mathbf{P}_m=yx+\sum_{i=0}^{m-1}c_ic_{i+1},\qquad
 K_m=\{yx\}\cup
 \{c_ic_{i+1},c_{i+1}c_i:0\leq i<m\}.
\]
The identity $xy\preceq \mathbf{P}_m$ is a renamed instance of
$\varepsilon_m$, and therefore holds in $\mathcal V$. We show that
$K_m$ is closed under all absorption rules arising from $\Delta_N$,
including arbitrary polynomial substitutions and left and right
word contexts. Every word in $K_m$ has length exactly two.

\textbf{Identity \eqref{lr:L1}.}
In either absorption identity obtained from this equality, every
premise word has length at least three under a nonempty
substitution. Thus neither rule can apply within $K_m$.

\textbf{Identity $t\preceq abc$.}
Every premise word has length at least three, so this rule cannot
apply either.

\textbf{Identity \eqref{lr:L3}.}
The premise is $\mathbf{Y}\mathbf{X}+\mathbf{X}$. If all its words, with contexts attached,
had length two, then, for any words $\mathbf{v}$ in $\mathbf{Y}$ and $\mathbf{u}$ in $\mathbf{X}$,
both $\boldsymbol{\ell} \mathbf{u} \mathbf{r}$ and $\boldsymbol{\ell} \mathbf{v}\mathbf{u} \mathbf{r}$ would have length two. Their lengths
differ by the positive integer $|\mathbf{v}|$, which is impossible.

\textbf{Identity \eqref{lr:L2}.}
The premise is $\mathbf{Y}\mathbf{X}+\mathbf{X}\mathbf{Z}$. Every product in the premise has length at least
two. Hence membership of all premise words in $K_m$ forces both
contexts to be empty and every word in $\mathbf{X},\mathbf{Y},\mathbf{Z}$ to be a letter. For any
conclusion word $ab$, the premise supplies
$ba\in K_m$ and, since $\mathbf{Z}$ is nonempty, some $ad\in K_m$. If $ba$
is an edge of the cycle, then $ab$ already belongs to $K_m$.
The only other possibility is $ba=yx$, so that $a=x$ and $b=y$.
But no word in $K_m$ starts with $x$, contradicting the required
membership of $ad$. Thus this rule adds only words already present.

\textbf{An identity $\varepsilon_n$ with $n\leq N$ odd.}
Every variable occurs in a quadratic premise word. The same length
argument forces both contexts to be empty and every polynomial
image to consist entirely of letters. An image may contain several
letters. Choose one letter $d_i$ from the image of each $z_i$. If all premise words
belong to $K_m$, then
\[
 d_1d_2,d_2d_3,\ldots,d_nd_1\in K_m.
\]
These edges form an odd closed walk of length $n$. As a directed
graph, $K_m$ consists of a bidirected cycle of length $m$ and a
disjoint single directed edge $y\to x$. The single edge lies in no
closed walk. The bidirected $m$-cycle has no odd closed walk of
length less than $m$: every odd closed walk in its underlying
undirected graph contains an odd simple cycle of no greater
length, and that graph's only simple cycle has length $m$.
Since $n\leq N<m$, we obtain a contradiction. Hence this rule cannot
apply.

These cases cover all polynomial substitutions, including images
containing several words, and all multiplicative contexts. Thus
$K_m$ is closed. It contains the support of $\mathbf{P}_m$
but not $xy$. By Lemma~\ref{lr:closure},
\begin{equation}
 \Delta_N\not\models xy\preceq \mathbf{P}_m.
 \label{lr:nonconsequence}
\end{equation}

Suppose that $\mathcal V$ had a finite basis $\Gamma$. Every identity
in $\Gamma$ is a consequence of the identities in
\eqref{lr:Vdefinition} and therefore follows from finitely
many of them. Since $\Gamma$ is finite, some $N$ satisfies
$\Delta_N\models\Gamma$. Every identity in $\Delta_N$ holds in
$\mathcal V=\operatorname{Mod}(\Gamma)$, so $\Delta_N$ would itself
define $\mathcal V$. This contradicts \eqref{lr:nonconsequence}.
\end{proof}

\begin{lem}[Quadratic absorption closure]\label{lr:quadratic-closure}
Let $\mathbf P$ have no summand of length at least three. Write its
support as $(U,E)$, where $U$ is the set of its singleton summands and
$E\subseteq X\times X$ is the set of its quadratic summands. Put
$A=U\cup\{a:(a,b)\in E\text{ for some }b\}$, and let $G$ be the
undirected graph on $A$ whose edges are the unordered pairs underlying
the members of $E\cap(A\times A)$; loops are allowed. Define
\[
 E^*=\begin{cases}
 E\cup\{(b,a):(a,b)\in E,\ b\in A\},&G\text{ is bipartite},\\
 E\cup\{(b,a):(a,b)\in E\},&G\text{ is not bipartite}.
 \end{cases}
\]
The words absorbed by $\mathbf P$ in $\mathcal V$ are exactly the
singleton words in $U$ and the quadratic words in $E^*$.
\end{lem}
\begin{proof}
Identity~\eqref{lr:L2} reverses an edge whose terminal vertex is the
initial vertex of another edge. Identity~\eqref{lr:L3} reverses an
edge whose terminal vertex belongs to $U$. Together they reverse
precisely the edges internal to $A$, without changing $A$, introducing
a new unordered pair, or changing $U$. Once these edges have been
reversed, a nonbipartite $G$ provides an odd directed closed walk.
A loop provides a closed walk of length three. Applying the
corresponding $\varepsilon_n$ reverses any chosen edge of $E$.
This proves that every listed word is absorbed.

For the converse, consider the support $(U,E^*)$. Every application
of \eqref{lr:L2} or \eqref{lr:L3} stays within that support. If $G$
is not bipartite, all its quadratic edges have already been reversed,
so an $\varepsilon_n$ adds nothing. If $G$ is bipartite, every
directed closed walk stays inside $A$, and every such walk has even
length. Thus no $\varepsilon_n$ can apply.

These observations include arbitrary polynomial substitutions and
multiplicative contexts. In a premise supported in lengths at most
two, a quadratic premise monomial forces each of its variables to
have an image consisting entirely of letters and forces both word
contexts to be empty. Every variable of \eqref{lr:L2},
\eqref{lr:L3}, or $\varepsilon_n$ occurs in a quadratic premise.
Selecting one letter from each image in an instance of
$\varepsilon_n$ gives an actual odd directed closed walk in $E^*$.
The premise of $t\preceq abc$ has length at least three, as do
the premises of \eqref{lr:L1}, so none applies. Hence the support
is closed under all the absorption rules. Lemma~\ref{lr:closure}
proves the assertion.
\end{proof}

We now give three small finite algebras that jointly generate
$\mathcal V$. Their construction also makes their operation tables
explicit. For a finite alphabet $X$ and a finite language
$K\subseteq X^+$, consider the ai-semiring of nonempty finite
languages, with union and concatenation. Identify $P$ and $Q$ when
\begin{equation}
 \boldsymbol\ell P\boldsymbol r\subseteq K
 \quad\Longleftrightarrow\quad
 \boldsymbol\ell Q\boldsymbol r\subseteq K
 \qquad(\boldsymbol\ell,\boldsymbol r\in X^*).
 \label{lr:syntactic-congruence}
\end{equation}
Empty contexts are allowed here; they are not constants in the
semiring signature. This is a congruence: union corresponds to the
conjunction of the two containment tests, while multiplication by a
finite language corresponds to a conjunction of tests with extended
contexts. Denote the quotient by $Q(K)$.

Take
\begin{align*}
 K_8&=\{p,q,pq,qp,pr,qr\},\qquad X_8=\{p,q,r\},\\
 K_6&=\{p,q,r,pp,qq,rr,pr,rp,qr,rq\},\qquad X_6=\{p,q,r\},\\
 K_4&=\{q,pp,pq,qp,qq\},\qquad X_4=\{p,q\},
\end{align*}
and write $Q_i=Q(K_i)$. Since these languages have no words of
length greater than two, only the empty context and contexts
consisting of one letter on one side can give a positive containment
test. Thus the classes can be read off from at most seven tests.
Intersecting the positive-test sets of the singleton languages
gives all classes, since a language passes a test exactly when each
of its words passes that test. The resulting algebras have
respectively eight, six and four elements.
In each table below, $\Omega$ is the class of $ppp$; the symbols
$p,q,r$ denote the corresponding singleton languages. The further
representatives are $s=p+q$, $u=p+r$, $v=q+r$, and $e=pq$ in $Q_8$;
in $Q_6$ they are $s=p+q$ and $e=pp$; in $Q_4$ one has $e=pp$.

\begin{center}
\small
\setlength{\tabcolsep}{3.2pt}
\begin{tabular}{c|cccccccc}
$+\ (Q_8)$&$p$&$q$&$r$&$s$&$u$&$v$&$e$&$\Omega$\\\hline
$p$&$p$&$s$&$u$&$s$&$u$&$\Omega$&$e$&$\Omega$\\
$q$&$s$&$q$&$v$&$s$&$\Omega$&$v$&$e$&$\Omega$\\
$r$&$u$&$v$&$r$&$\Omega$&$u$&$v$&$\Omega$&$\Omega$\\
$s$&$s$&$s$&$\Omega$&$s$&$\Omega$&$\Omega$&$e$&$\Omega$\\
$u$&$u$&$\Omega$&$u$&$\Omega$&$u$&$\Omega$&$\Omega$&$\Omega$\\
$v$&$\Omega$&$v$&$v$&$\Omega$&$\Omega$&$v$&$\Omega$&$\Omega$\\
$e$&$e$&$e$&$\Omega$&$e$&$\Omega$&$\Omega$&$e$&$\Omega$\\
$\Omega$&$\Omega$&$\Omega$&$\Omega$&$\Omega$&$\Omega$&$\Omega$&$\Omega$&$\Omega$
\end{tabular}

\medskip
\begin{tabular}{c|cccccccc}
$\cdot\ (Q_8)$&$p$&$q$&$r$&$s$&$u$&$v$&$e$&$\Omega$\\\hline
$p$&$\Omega$&$e$&$e$&$\Omega$&$\Omega$&$e$&$\Omega$&$\Omega$\\
$q$&$e$&$\Omega$&$e$&$\Omega$&$e$&$\Omega$&$\Omega$&$\Omega$\\
$r$&$\Omega$&$\Omega$&$\Omega$&$\Omega$&$\Omega$&$\Omega$&$\Omega$&$\Omega$\\
$s$&$\Omega$&$\Omega$&$e$&$\Omega$&$\Omega$&$\Omega$&$\Omega$&$\Omega$\\
$u$&$\Omega$&$\Omega$&$\Omega$&$\Omega$&$\Omega$&$\Omega$&$\Omega$&$\Omega$\\
$v$&$\Omega$&$\Omega$&$\Omega$&$\Omega$&$\Omega$&$\Omega$&$\Omega$&$\Omega$\\
$e$&$\Omega$&$\Omega$&$\Omega$&$\Omega$&$\Omega$&$\Omega$&$\Omega$&$\Omega$\\
$\Omega$&$\Omega$&$\Omega$&$\Omega$&$\Omega$&$\Omega$&$\Omega$&$\Omega$&$\Omega$
\end{tabular}

\medskip
\begin{tabular}{c|cccccc}
$+\ (Q_6)$&$p$&$q$&$r$&$s$&$e$&$\Omega$\\\hline
$p$&$p$&$s$&$p$&$s$&$e$&$\Omega$\\
$q$&$s$&$q$&$q$&$s$&$e$&$\Omega$\\
$r$&$p$&$q$&$r$&$s$&$e$&$\Omega$\\
$s$&$s$&$s$&$s$&$s$&$e$&$\Omega$\\
$e$&$e$&$e$&$e$&$e$&$e$&$\Omega$\\
$\Omega$&$\Omega$&$\Omega$&$\Omega$&$\Omega$&$\Omega$&$\Omega$
\end{tabular}
\qquad
\begin{tabular}{c|cccccc}
$\cdot\ (Q_6)$&$p$&$q$&$r$&$s$&$e$&$\Omega$\\\hline
$p$&$e$&$\Omega$&$e$&$\Omega$&$\Omega$&$\Omega$\\
$q$&$\Omega$&$e$&$e$&$\Omega$&$\Omega$&$\Omega$\\
$r$&$e$&$e$&$e$&$e$&$\Omega$&$\Omega$\\
$s$&$\Omega$&$\Omega$&$e$&$\Omega$&$\Omega$&$\Omega$\\
$e$&$\Omega$&$\Omega$&$\Omega$&$\Omega$&$\Omega$&$\Omega$\\
$\Omega$&$\Omega$&$\Omega$&$\Omega$&$\Omega$&$\Omega$&$\Omega$
\end{tabular}

\medskip
\begin{tabular}{c|cccc}
$+\ (Q_4)$&$p$&$q$&$e$&$\Omega$\\\hline
$p$&$p$&$p$&$\Omega$&$\Omega$\\
$q$&$p$&$q$&$e$&$\Omega$\\
$e$&$\Omega$&$e$&$e$&$\Omega$\\
$\Omega$&$\Omega$&$\Omega$&$\Omega$&$\Omega$
\end{tabular}
\qquad
\begin{tabular}{c|cccc}
$\cdot\ (Q_4)$&$p$&$q$&$e$&$\Omega$\\\hline
$p$&$e$&$e$&$\Omega$&$\Omega$\\
$q$&$e$&$e$&$\Omega$&$\Omega$\\
$e$&$\Omega$&$\Omega$&$\Omega$&$\Omega$\\
$\Omega$&$\Omega$&$\Omega$&$\Omega$&$\Omega$
\end{tabular}
\end{center}

\begin{lem}\label{lr:finite-generator}
The $0$-direct sum $B=Q_8\oplus_0Q_6\oplus_0Q_4$ has sixteen
elements and generates $\mathcal V$.
\end{lem}
\begin{proof}
The symbol $0$ in ``$0$-direct sum'' denotes the common multiplicative
zero $\Omega$, which is additively greatest, not an additive identity.
Explicitly, identify the three copies of $\Omega$, retain the two
operations within each component, and let both operations on elements
from different components have value $\Omega$. This is the subalgebra
of $Q_8\times Q_6\times Q_4$ consisting of tuples with at most one
coordinate different from $\Omega$. Its projections onto the three
factors are surjective. Therefore it is an ai-semiring and
\[
 |B|=8+6+4-2=16,\qquad
 \V(B)=\V(Q_8,Q_6,Q_4).
\]

Both $Q_6$ and $Q_4$ are commutative, and all their triple products
are $\Omega$, so they belong to $\mathcal C$. In $Q_8$, all products
are either $e$ or $\Omega$, and its non-$\Omega$ products are precisely
\[
 pq=pr=pv=qp=qr=qu=sr=e.
\]
The only directed cycle in the graph of these non-$\Omega$
products is $p\leftrightarrow q$.
An odd cyclic sequence therefore has at least one consecutive product
equal to $\Omega$. Every $\varepsilon_n$ holds because its right side
then evaluates to $\Omega$. All triple products are $\Omega$.

To check \eqref{lr:L2}, it is enough to consider assignments with
$yx=xz=e$. Then $x$ has both an incoming and an outgoing edge in
the displayed graph, so $x\in\{p,q\}$, and $y$ is the other one.
It follows that $xy=e$. To check \eqref{lr:L3}, assume $yx=e$.
If $x\in\{p,q\}$, again $xy=e$. Otherwise
$x\in\{r,u,v\}$ and $e+x=\Omega$. This proves both absorptions.
Thus $Q_8\in\mathcal V$ and $\V(B)\subseteq\mathcal V$.

For the reverse inclusion, it suffices to separate every absorption
$\mathbf w\preceq\mathbf P$ that fails in $\mathcal V$ in one of
the three factors. The polynomial $\mathbf P$ has no word of length
at least three, since such a word would make its value greatest.
Use the notation of Lemma~\ref{lr:quadratic-closure} and replace $E$
by its absorption closure.

If $\mathbf w=x\notin U$ is a singleton, assign $x=p$ and every
other variable $q$ in $Q_4$. All quadratic words have value $e$,
and all singleton summands of $\mathbf P$ have value $q$. Thus
$\mathbf P$ has value $q$ or $e$, neither of which absorbs $p$.
If $|\mathbf w|\geq3$, assign every variable $r$ in $Q_6$.
The target is $\Omega$, while the right side is $r$ or $e$.

Suppose next that $\mathbf w=xy$ and its unordered pair is absent
from $E$. If $x\ne y$, assign $x=p$, $y=q$, and every other variable
$r$ in $Q_6$. The target is $\Omega$. Every right-side quadratic
word evaluates to $e$, and all singleton values are at most $e$;
the right side therefore cannot absorb $\Omega$. If $x=y$, assign
$x=s$ and all other variables $r$. Now $s^2=\Omega$, whereas
$sr=rs=r^2=e$, giving the same conclusion.

The only remaining case is $xy\notin E$ but $yx\in E$.
The closure description implies $x\notin A$, $y\in A$, and that
the graph induced on $A$ is bipartite. Choose a proper two-coloring
of this graph, assign its two color classes the values $p,q$ in
$Q_8$, and assign every variable outside $A$ the value $r$.
Every right-side quadratic word evaluates to $e$: internal edges
join different colors, and the other edges go from $A$ to its
complement. Singleton summands have values $p$ or $q$, both at most
$e$. Since $yx$ occurs, the right side is exactly $e$.
The target is $rp$ or $rq$, both equal to $\Omega$.

All possible failures have been separated. Since every identity
is equivalent over $\AI$ to a finite family of word absorptions,
$\Id(B)=\Id(\mathcal V)$ and $\V(B)=\mathcal V$.
\end{proof}

\begin{proof}[Proof of Theorem~\ref{lr:counterexample}]
Take $S=B$ from Lemma~\ref{lr:finite-generator}. By
Lemma~\ref{lr:Vnfb}, this sixteen-element ai-semiring is nonfinitely
based. On the other hand, \eqref{lr:join} gives
\[
 \V(B\times L_2)=\mathcal V\vee\V(L_2)=\mathcal W,
\]
whose finite basis is supplied by Lemma~\ref{lr:Wbasis}.

Reversing multiplication gives a bijection between the identities
of an ai-semiring and those of its opposite and preserves equational
derivability. Hence it preserves finite and nonfinite basability.
Moreover,
\[
 (B\times L_2)^{\mathrm{op}}\cong B^{\mathrm{op}}\times R_2.
\]
Thus the sixteen-element algebra $B^{\mathrm{op}}$ gives the
corresponding reverse-direction counterexample for $R_2$.
\end{proof}

\begin{remark}
The two witnesses have sixteen elements. No minimality claim is made.
\end{remark}

\subsection{The null factor: a ten-element counterexample to the reverse implication}
\label{n:section}

The forward implication fails by Theorem~\ref{nf:forward-counterexample}.
For the reverse implication, we construct a ten-element commutative
ai-semiring $C_{10}$. The same source also gives the reverse
$D_2$ counterexample in Theorem~\ref{d:reversecounter}.
Throughout this subsection, $\infty$ denotes an additive greatest element
that is also a multiplicative zero; it is not a constant symbol in the
signature.

\begin{thm}\label{n:reverse-counterexample}
The ten-element commutative ai-semiring $C_{10}$ constructed below satisfies
\[
\neg\FB(C_{10}),\qquad \FB(C_{10}\times N_2).
\]
Every product of three elements of $C_{10}$ is $\infty$. Moreover,
$C_{10}\times N_2$ has the following finite identity basis relative to $\AI$:
\begin{equation}\label{n:finite-product-basis}
xy\approx yx,\qquad xyz\approx uvw,\qquad xy+uvw\approx uvw.
\end{equation}
\end{thm}

We first describe a nonfinitely based variety by its exact free-algebra
normal forms. We then show that it is generated by an explicitly defined
ten-element semiring.

\subsubsection{A variety defined by odd-cycle identities}

Let $\mathcal C$ be the variety defined by
\begin{equation}\label{n:C-three-nil}
\AI,\qquad xy\approx yx,\qquad t\preceq xyz.
\end{equation}
The last identity makes every triple product an additive greatest
element. All triple products therefore have the same value, and this
value is a multiplicative zero. Thus every word of length at least three
has value $\infty$.

For every odd integer $n\geq3$, put
\[
\mathbf Q_n=z_1z_2+z_2z_3+\cdots+z_{n-1}z_n+z_nz_1
\]
and define
\begin{equation}\tag{$\varepsilon_n$}\label{n:odd-cycle-law}
x\preceq xy+\mathbf Q_n,
\end{equation}
where $x,y,z_1,\ldots,z_n$ are distinct variables. Let $\mathcal U$ be
the subvariety of $\mathcal C$ defined by all these identities.

For a commutative polynomial with no word of length at least three,
write $A$ for its set of singleton-word summands and $E$ for its set of
quadratic summands. Regard $E$ as an undirected graph with loops: $uv$
is an edge and $u^2$ is a loop. Write $\operatorname{inc}(E)$ for the
set of vertices incident to an edge, including loops, and define
\[
\operatorname{cl}_E(A)=
\begin{cases}
A\cup\operatorname{inc}(E),&E\text{ is nonbipartite},\\
A,&E\text{ is bipartite}.
\end{cases}
\]
In particular, a graph with a loop is nonbipartite.

\begin{lem}\label{n:finite-normal-forms}
The free $\mathcal U$-algebra on a nonempty set $X$ has the following
normal forms:
\begin{enumerate}[label=\textup{(\roman*)}]
\item a single element $\infty$, representing every polynomial with a
word of length at least three;
\item pairs $(A,E)$, where $A$ and $E$ are not both empty,
$\operatorname{cl}_E(A)=A$, and $A$ is a finite subset of $X$ while $E$ is a finite set of
unordered pairs from $X$, with repetitions allowed within a pair.
\end{enumerate}
More precisely, a polynomial with singleton support $A$ and quadratic
support $E$ has normal form $(\operatorname{cl}_E(A),E)$.
\end{lem}

\begin{proof}
First consider $\mathcal C$. Its free algebra consists of $\infty$
and all nonempty pairs $(A,E)$ without the closure condition. Addition
is union of the two components, with $\infty$ absorbing. A product is
$\infty$ if either factor is $\infty$ or has a nonempty quadratic
component. Otherwise, the product of $(A,\varnothing)$ and
$(B,\varnothing)$ has empty singleton support and edge set
\[
\{uv:u\in A,\ v\in B\}.
\]
These operations satisfy \eqref{n:C-three-nil}. Conversely,
distributivity and \eqref{n:C-three-nil} reduce every term to the stated
form, so this description is exact.

On this free algebra identify $(A,E)$ and $(B,E)$ precisely when
\[
\operatorname{cl}_E(A)=\operatorname{cl}_E(B).
\]
This is a congruence. For addition, nonbipartiteness persists when edges
are added, and
$\operatorname{inc}(E)\subseteq\operatorname{inc}(E\cup F)$.
Consequently, every singleton discarded by the relation is still
discarded after adding another pair. For multiplication, any change of
singleton support occurs with a nonempty quadratic component, and
every product of such an element with a nonempty term is $\infty$.

The quotient satisfies every identity \eqref{n:odd-cycle-law}.
Consider an arbitrary substitution of polynomials. If the right side
has value $\infty$, the inequality is automatic. Otherwise every term
substituted for $x,y,z_1,\ldots,z_n$ is a nonempty sum of singleton
words: a quadratic summand in any such term would produce a word of
length at least three in $xy$ or in an incident cycle edge. Choose one
letter from the term substituted for each $z_i$. The corresponding
cycle edges form an odd closed walk in the quadratic graph of the right
side. That graph is therefore nonbipartite. Every letter substituted
for $x$ is incident to an edge contributed by $xy$, so the closure
inserts all the singleton summands required on the left.

Conversely, suppose $E$ is nonbipartite. It has an odd closed walk of
some length $n\geq3$; a loop may be traversed three times. Substitute
the successive vertices of this walk for the $z_i$ in
\eqref{n:odd-cycle-law}. For any $r\in\operatorname{inc}(E)$, choose
an incident edge $rs$ and substitute $x=r$, $y=s$. All right-side
edges already belong to $E$. Adding the original polynomial to the
resulting absorption inserts the singleton $r$. Doing this for every
incident vertex gives precisely $\operatorname{cl}_E(A)$.
Thus the displayed congruence gives exactly the identifications made
by the defining identities of $\mathcal U$.
\end{proof}

\begin{lem}\label{n:finite-variety-nfb}
The variety $\mathcal U$ is nonfinitely based.
\end{lem}

\begin{proof}
Suppose that a finite set $\Delta$ is a basis for $\mathcal U$.
Each identity in $\Delta$ has a finite derivation from
\eqref{n:C-three-nil} and the family \eqref{n:odd-cycle-law}.
Hence all of $\Delta$ follows from \eqref{n:C-three-nil} and a finite
set of identities $\varepsilon_j$, with odd indices in a finite set
$J$. Choose an odd $n\geq3$ larger than every member of $J$; if $J$ is
empty, take $n=3$.

The normal-form congruence in the preceding proof remains valid if
``$E$ is nonbipartite'' is replaced by ``$E$ has an odd closed walk of
length belonging to $J$.'' Indeed, this property persists when edges
are added, and a substituted instance of $\varepsilon_j$ produces
a closed walk of exactly length $j$, by the same letter-selection
argument. The resulting quotient therefore satisfies
\eqref{n:C-three-nil} and all the selected $\varepsilon_j$.

Now work over the variables of $\varepsilon_n$. The graph of
$xy+\mathbf Q_n$ is the disjoint union of a single edge $xy$ and a
simple cycle of length $n$. It has no odd closed walk of length less
than $n$: the edge component is bipartite, and the shortest odd
closed walk in the cycle component has length $n$. Thus its singleton
support remains empty in the quotient just constructed. Adding $x$
changes its normal form, so $\varepsilon_n$ fails there.
It follows that $\varepsilon_n$ is not a consequence of the selected
finite axioms, and hence not a consequence of $\Delta$. This
contradicts its validity in $\mathcal U$.
\end{proof}

\subsubsection{An explicit ten-element generator}

Define a six-element algebra $Q$ on
\[
\{c,p,q,d,w,\infty\}.
\]
Addition is the join operation in the order
\[
c<p<d<w<\infty,\qquad c<q<d,
\]
where $p,q$ are incomparable; in particular, $p+q=d$.
Multiplication is given by
\[
\begin{array}{c|cccccc}
\cdot&c&p&q&d&w&\infty\\\hline
c&w&w&w&w&\infty&\infty\\
p&w&w&\infty&\infty&\infty&\infty\\
q&w&\infty&w&\infty&\infty&\infty\\
d&w&\infty&\infty&\infty&\infty&\infty\\
w&\infty&\infty&\infty&\infty&\infty&\infty\\
\infty&\infty&\infty&\infty&\infty&\infty&\infty
\end{array}
\]

Define a five-element algebra $B$ on
\[
\{a,b,c,w,\infty\}.
\]
Its addition and multiplication are given by
\[
\begin{array}{c|ccccc}
+&a&b&c&w&\infty\\\hline
a&a&a&\infty&\infty&\infty\\
b&a&b&w&w&\infty\\
c&\infty&w&c&w&\infty\\
w&\infty&w&w&w&\infty\\
\infty&\infty&\infty&\infty&\infty&\infty
\end{array}
\qquad
\begin{array}{c|ccccc}
\cdot&a&b&c&w&\infty\\\hline
a&\infty&\infty&w&\infty&\infty\\
b&\infty&\infty&w&\infty&\infty\\
c&w&w&\infty&\infty&\infty\\
w&\infty&\infty&\infty&\infty&\infty\\
\infty&\infty&\infty&\infty&\infty&\infty
\end{array}
\]
Thus the only products in $B$ different from $\infty$ are
$ac=ca=bc=cb=w$.

\begin{lem}\label{n:small-components}
The algebras $Q$ and $B$ belong to $\mathcal U$.
\end{lem}

\begin{proof}
For $Q$, identify $c,p,q,d$ with the subsets
$\varnothing,\{P\},\{Q\},\{P,Q\}$ of a two-element set.
Their addition is union. The product of two such subsets is
$\infty$ if one contains $P$ and the other contains $Q$, in either
order; otherwise it is $w$. The forbidden-pair condition distributes
over union, proving distributivity on these four elements. The cases
with $w$ or $\infty$ follow because these elements lie above all four
and multiply every element to $\infty$. Every triple product is
$\infty$, proving multiplicative associativity. The remaining
ai-semiring laws follow from the displayed join order.

For $B$, addition is the join operation in the order
$b<a<\infty$ and $b,c<w<\infty$, with $a$ incomparable with $c,w$.
Multiplication by $c$ sends $a,b$ to $w$ and $c,w,\infty$ to
$\infty$; this map preserves every join in the displayed addition
table. Multiplication by $a$ or $b$ sends $c$ to $w$ and every other
element to $\infty$, which also preserves all these joins.
Multiplication by $w$ or $\infty$ is constant. Thus distributivity
holds. Every triple product is $\infty$, so associativity holds as
well. Both algebras therefore satisfy \eqref{n:C-three-nil}.

In $Q$, if the value of $x$ belongs to $\{c,p,q,d\}$, it is at most
$w$, whereas every quadratic word has value $w$ or $\infty$.
If $x$ has value $w$ or $\infty$, the guard $xy$ has value
$\infty$. This proves every $\varepsilon_n$ in $Q$.

In $B$, the value of $\mathbf Q_n$ is always $\infty$. Otherwise
every cycle-edge product would be $w$. The successive vertex values
would then have to alternate between $\{a,b\}$ and $\{c\}$,
which is impossible on an odd cycle. A vertex value $w$ or $\infty$
already makes an incident product $\infty$. Hence all
$\varepsilon_n$ hold in $B$ as well.
\end{proof}

Take disjoint copies of $Q$ and $B$, identify their elements
$\infty$, retain the operations within each component, and define
both the sum and product of elements from different components to be
$\infty$. Denote the resulting $0$-direct sum by
\begin{equation}\label{n:S-definition}
C_{10}=Q\oplus_0 B.
\end{equation}
The symbols $c,w$ in the two components denote distinct elements of
$C_{10}$; only the two copies of $\infty$ are identified. Therefore
$|C_{10}|=6+5-1=10$.

Each component embeds in $C_{10}$, and $C_{10}$ embeds in $Q\times B$ by sending
an element to itself in its own coordinate and to $\infty$ in the
other coordinate. The image consists of the two coordinate copies
with common element $(\infty,\infty)$ and is closed under both
operations. In particular, $C_{10}$ is an ai-semiring and
\[
\operatorname{Id}(C_{10})
=\operatorname{Id}(Q)\cap\operatorname{Id}(B).
\]

\begin{lem}\label{n:finite-generator}
The ten-element semiring $C_{10}$ in \eqref{n:S-definition} generates
$\mathcal U$.
\end{lem}

\begin{proof}
By Lemma~\ref{n:small-components}, $\V(C_{10})\subseteq\mathcal U$.
For the reverse inclusion, we separate any two distinct normal forms
of Lemma~\ref{n:finite-normal-forms} by an assignment in $Q$ or $B$.

If one form is $\infty$ and the other is not, assign every variable
$c$ in $Q$. The form without cubic words takes value $c$ or $w$,
whereas the form with a cubic takes value $\infty$.

Suppose next that the quadratic supports differ. Choose an edge in
the first support that is absent from the second. If it is $xy$ with
$x\ne y$, assign $x=p$, $y=q$ and all other variables $c$ in $Q$.
Precisely that quadratic edge has value $\infty$; every other
quadratic word has value $w$, and all singleton values are at most
$w$. Thus the first form has value $\infty$ and the second has value
at most $w$. If the disputed edge is $x^2$, assign $x=d$ and all other
variables $c$; the same argument applies.

It remains to consider equal quadratic supports $E$ but distinct
closed singleton supports. Choose $x$ in the first closed singleton
support and not in the second. If $E$ is bipartite, choose a
bipartition with $x$ in its first part, assigning isolated vertices
arbitrarily. In $B$, assign $x=a$, all other variables in its part
$b$, and all variables in the second part $c$. Every quadratic edge
has value $w$. The second form has value in $\{b,c,w\}$, since it
does not contain the singleton $x$; none of these three elements is
above $a$. The first form is above $a$, so their values differ.

If $E$ is nonbipartite, the absence of $x$ from the second closed
singleton support implies $x\notin\operatorname{inc}(E)$.
Assign $x=\infty$ and every other variable $c$ in $Q$. The second
form has value $w$, because $E$ is nonempty and no word of that form
contains $x$. The first form has value $\infty$. This completes all
cases, so $C_{10}$ and $\mathcal U$ have exactly the same identities.
\end{proof}

\subsubsection{A finite basis after adjoining the null factor}

\begin{thm}\label{n:join-finite-basis}
The variety $\mathcal U\vee\V(N_2)$ is defined by $\AI$ and
\eqref{n:finite-product-basis}.
\end{thm}

\begin{proof}
Let $\mathcal W$ be the variety defined by those identities. They
hold in both $\mathcal U$ and $N_2$, so
$\mathcal U\vee\V(N_2)\subseteq\mathcal W$.

The identities of $\mathcal W$ give the following exact polynomial
normal forms. A polynomial without a cubic word is described by its
singleton support $A$ and its quadratic support $E$. A polynomial
with a cubic word is described by $(A,\top)$: all words of length
at least three are equal, and their common value absorbs every
quadratic summand, but no singleton summand is removed. These forms
are exact, as can also be checked by defining addition to take unions
of singleton supports and unions of quadratic supports, with $\top$
absorbing only in the nonlinear component. A product has empty
singleton support; its nonlinear component is the usual set of cross
edges when both factors have only singleton summands, and is $\top$
if either factor has a quadratic or cubic summand. The resulting
algebra satisfies the proposed basis and separates the displayed
forms.

An identity holds in $N_2$ precisely when its two sides have the same
singleton supports, because every longer word evaluates to the least
element. Intersect this condition with the normal-form description
for $\mathcal U$. For two polynomials without cubic words, equality
in $\mathcal U$ requires equal quadratic supports and equal closed
singleton supports. If the original singleton supports are already
equal, this is exactly equality of $A$ and $E$. For two polynomials
with cubic words, equality in $\mathcal U$ is automatic, so only
equality of singleton supports remains. A polynomial with a cubic
word is never equal in $\mathcal U$ to one without a cubic word.
These are precisely the three equality conditions for the
$\mathcal W$ normal forms. Hence
\[
\operatorname{Id}(\mathcal U)\cap\operatorname{Id}(N_2)
=\operatorname{Id}(\mathcal W),
\]
which proves the reverse inclusion and the claimed basis.
\end{proof}

\begin{proof}[Proof of Theorem~\ref{n:reverse-counterexample}]
Lemmas~\ref{n:finite-variety-nfb} and \ref{n:finite-generator} show that
$C_{10}$ is nonfinitely based.
On the other hand,
\[
\V(C_{10}\times N_2)=\V(C_{10})\vee\V(N_2)
=\mathcal U\vee\V(N_2),
\]
which has the finite basis in Theorem~\ref{n:join-finite-basis}.
Both components of $C_{10}$, and hence their $0$-direct sum, are
commutative and have every triple product equal to $\infty$.
\end{proof}

\begin{remark}
The same ten-element source is used for the reverse implications for
$N_2$ and $D_2$. No minimality assertion is made for its order.
\end{remark}

\begin{proof}[Proof of Theorem~\ref{intro:main}]
The universal equivalences are Theorems~\ref{m:equivalence} and
\ref{t:equivalence}. The forward-direction counterexamples are
Theorems~\ref{nf:forward-counterexample},
\ref{lf:forward-counterexample} and \ref{d:forward-counterexample}.
The reverse-direction counterexamples are
Theorems~\ref{d:reversecounter}, \ref{lr:counterexample} and
\ref{n:reverse-counterexample}. These theorems also establish all
the stated structural properties of the witnesses.
\end{proof}

\section{Structural classes with the equivalence}

The unrestricted classification leaves room for positive results on
structurally defined classes. The elementary observation below gives
one such class for each of $L_2,R_2,D_2,N_2$.

\begin{lem}\label{recovery:absorption}
Let $S$ and $E$ be algebras of the same signature. If $E\in\V(S)$, then
\[
 \V(S\times E)=\V(S).
\]
Consequently, $S\times E$ is finitely based if and only if $S$ is
finitely based. In particular, this conclusion holds if $E$ is
isomorphic to a subalgebra or a homomorphic image of $S$.
\end{lem}
\begin{proof}
Every identity of a direct product holds in each factor, and conversely.
Thus
\[
 \Id(S\times E)=\Id(S)\cap\Id(E).
\]
The hypothesis $E\in\V(S)$ gives $\Id(S)\subseteq\Id(E)$, so the
intersection is $\Id(S)$. Hence $S$ and $S\times E$ have precisely the
same identities and, in particular, the same finite basis property.
\end{proof}

\begin{remark}[Structural classes on which equivalence is restored]
\label{recovery:structural}
Although the unrestricted equivalence fails for the four factors under
consideration, it holds on the following explicit structural classes.
All algebras are still considered in the constant-free signature
$(+,\cdot)$; elements denoted by $0$ and $1$ are not added to the signature.
\begin{enumerate}
\item For $E=L_2$, it is enough that the multiplicative reduct of $S$
contain a nontrivial left-zero subsemigroup.
\item For $E=R_2$, it is enough that the multiplicative reduct of $S$
contain a nontrivial right-zero subsemigroup.
\item For $E=D_2$, every nontrivial unital ai-semiring with a semiring
zero has the required property. Here a semiring zero is an additive
identity that is absorbing for multiplication, and the multiplicative
identity is distinct from it.
\item For $E=N_2$, every nontrivial nil ai-semiring with a semiring zero
has the required property. More generally, it suffices that a semiring
zero exist and that the multiplicative reduct contain a nonzero
nilpotent element.
\end{enumerate}
These sufficient conditions are structural; no finite identity basis
or additional identity of $S$ is assumed. They are not asserted to be
necessary or maximal.

Here are direct proofs. In the first case, choose distinct elements
$a,b$ of a left-zero subsemigroup, so
\[
 a^2=a,\qquad ab=a,\qquad ba=b,\qquad b^2=b.
\]
Put $c=a+b$. Distributivity gives
\[
 ac=a,\quad ca=c,\quad bc=b,\quad cb=c,\quad c^2=c.
\]
At least one of $a,b$ differs from $c$; denote such an element by $d$.
The set $\{d,c\}$ is a subsemiring, with $d+c=c$ and multiplication
equal to the left projection. It is therefore isomorphic to $L_2$.
The proof for $R_2$ follows by reversing multiplication.

For the third case, $\{0,1\}$ is a subsemiring whose addition is the
maximum and whose multiplication is the minimum on $0<1$. Thus it is
isomorphic to $D_2$.

For the fourth case, let $a\ne0$ be nilpotent and choose the least
$k\geq2$ such that $a^k=0$. Put $b=a^{k-1}$. Then $b\ne0$ and
$b^2=a^{2k-2}=0$, since $2k-2\geq k$. Hence $\{0,b\}$ is a
subsemiring with $0+b=b$ and constant multiplication equal to $0$,
and is isomorphic to $N_2$. In all four cases,
Lemma~\ref{recovery:absorption} proves the claimed equivalence.
\end{remark}


\begin{thebibliography}{99}

\bibitem{BS1981}
Burris, S.N., Sankappanavar, H.P.:
A Course in Universal Algebra.
Graduate Texts in Mathematics, vol.~78.
Springer-Verlag, New York (1981).

\bibitem{Dolinka2007}
Dolinka, I.:
A nonfinitely based finite semiring.
Int. J. Algebra Comput. \textbf{17}(8), 1537--1551 (2007).
\href{https://doi.org/10.1142/S0218196707004177}{\nolinkurl{doi:10.1142/S0218196707004177}}.

\bibitem{Dolinka2009INFB}
Dolinka, I.:
A class of inherently nonfinitely based semirings.
Algebra Univers. \textbf{60}, 19--35 (2009).
\href{https://doi.org/10.1007/s00012-008-2084-y}{\nolinkurl{doi:10.1007/s00012-008-2084-y}}.

\bibitem{Dolinka2009Remark}
Dolinka, I.:
A remark on nonfinitely based semirings.
Semigroup Forum \textbf{78}, 368--373 (2009).
\href{https://doi.org/10.1007/s00233-008-9096-y}{\nolinkurl{doi:10.1007/s00233-008-9096-y}}.

\bibitem{Dolinka2009End}
Dolinka, I.:
The finite basis problem for endomorphism semirings of finite semilattices with zero.
Algebra Univers. \textbf{61}(3--4), 441--448 (2009).
\href{https://doi.org/10.1007/s00012-009-0024-0}{\nolinkurl{doi:10.1007/s00012-009-0024-0}}.

\bibitem{DolinkaGusevVolkov2025End}
Dolinka, I., Gusev, S.V., Volkov, M.V.:
The finite basis problem for the endomorphism semirings of finite semilattices.
Bull. Belg. Math. Soc. Simon Stevin \textbf{32}(5), 657--674 (2025).
\href{https://doi.org/10.36045/j.bbms.250801}{\nolinkurl{doi:10.36045/j.bbms.250801}}.

\bibitem{Gao2026}
Gao, Z.D.:
A finitely based finite semiring generates a variety with continuum many subvarieties.
Preprint, arXiv:2603.00893 (2026).
\href{https://arxiv.org/abs/2603.00893}{\nolinkurl{arXiv:2603.00893}}.

\bibitem{Graczynska1983}
Graczy\'{n}ska, E.:
On regular identities.
Algebra Univers. \textbf{17}, 369--375 (1983).
\href{https://doi.org/10.1007/BF01194544}{\nolinkurl{doi:10.1007/BF01194544}}.

\bibitem{GusevVolkov2023Power}
Gusev, S.V., Volkov, M.V.:
Semiring and involution identities of power groups.
J. Aust. Math. Soc. \textbf{115}(3), 354--374 (2023).
\href{https://doi.org/10.1017/S1446788722000374}{\nolinkurl{doi:10.1017/S1446788722000374}}.

\bibitem{GusevVolkov2026Chains}
Gusev, S.V., Volkov, M.V.:
The finite basis problem for endomorphism semirings of finite chains.
Int. J. Algebra Comput. (2026), published online.
\href{https://doi.org/10.1142/S0218196726400084}{\nolinkurl{doi:10.1142/S0218196726400084}}.

\bibitem{HW1998}
Hebisch, U., Weinert, H.J.:
Semirings: Algebraic Theory and Applications in Computer Science.
Series in Algebra, vol.~5.
World Scientific, Singapore (1998).
\href{https://doi.org/10.1142/3903}{\nolinkurl{doi:10.1142/3903}}.

\bibitem{JRZ}
Jackson, M., Ren, M.M., Zhao, X.Z.:
Nonfinitely based ai-semirings with finitely based semigroup reducts.
J. Algebra \textbf{611}, 211--245 (2022).
\href{https://doi.org/10.1016/j.jalgebra.2022.07.042}{\nolinkurl{doi:10.1016/j.jalgebra.2022.07.042}}.

\bibitem{KP2005}
Ku\v{r}il, M., Pol\'{a}k, L.:
On varieties of semilattice-ordered semigroups.
Semigroup Forum \textbf{71}, 27--48 (2005).
\href{https://doi.org/10.1007/s00233-004-0176-3}{\nolinkurl{doi:10.1007/s00233-004-0176-3}}.

\bibitem{Pastijn2006}
Pastijn, F.:
Weak commutativity in idempotent semirings.
Semigroup Forum \textbf{72}, 283--311 (2006).
\href{https://doi.org/10.1007/s00233-005-0543-8}{\nolinkurl{doi:10.1007/s00233-005-0543-8}}.

\bibitem{RJZL}
Ren, M.M., Jackson, M., Zhao, X.Z., Lei, D.L.:
Flat extensions of groups and limit varieties of additively idempotent semirings.
J. Algebra \textbf{623}, 64--85 (2023).
\href{https://doi.org/10.1016/j.jalgebra.2023.01.027}{\nolinkurl{doi:10.1016/j.jalgebra.2023.01.027}}.

\bibitem{RLZC2025}
Ren, M.M., Liu, J.Y., Zeng, L.L., Chen, M.L.:
The finite basis problem for additively idempotent semirings of order four, I.
Semigroup Forum \textbf{110}(2), 422--457 (2025).
\href{https://doi.org/10.1007/s00233-025-10520-7}{\nolinkurl{doi:10.1007/s00233-025-10520-7}}.

\bibitem{RLYC2026}
Ren, M.M., Liu, Z.X., Yue, M.Y., Chen, Y.Z.:
The finite basis problem for additively idempotent semirings of order four, III.
Semigroup Forum \textbf{112}(2), 541--573 (2026).
\href{https://doi.org/10.1007/s00233-025-10584-5}{\nolinkurl{doi:10.1007/s00233-025-10584-5}}.

\bibitem{RenZeng2019}
Ren, M.M., Zeng, L.L.:
On a hereditarily finitely based ai-semiring variety.
Soft Comput. \textbf{23}, 6819--6825 (2019).
\href{https://doi.org/10.1007/s00500-018-03719-0}{\nolinkurl{doi:10.1007/s00500-018-03719-0}}.

\bibitem{RZW2017}
Ren, M.M., Zhao, X.Z., Wang, A.F.:
On the varieties of ai-semirings satisfying $x^3\approx x$.
Algebra Univers. \textbf{77}(4), 395--408 (2017).
\href{https://doi.org/10.1007/s00012-017-0438-z}{\nolinkurl{doi:10.1007/s00012-017-0438-z}}.

\bibitem{SR}
Shao, Y., Ren, M.M.:
On the varieties generated by ai-semirings of order two.
Semigroup Forum \textbf{91}(1), 171--184 (2015).
\href{https://doi.org/10.1007/s00233-014-9667-z}{\nolinkurl{doi:10.1007/s00233-014-9667-z}}.

\bibitem{Shaprynskii2024}
Shaprynski\u{\i}, V.Yu.:
Semiring identities of the semigroup $B_0$.
Semigroup Forum \textbf{109}(3), 693--705 (2024).
\href{https://doi.org/10.1007/s00233-024-10476-0}{\nolinkurl{doi:10.1007/s00233-024-10476-0}}.

\bibitem{Volkov2021Brandt}
Volkov, M.V.:
Semiring identities of the Brandt monoid.
Algebra Univers. \textbf{82}(3), article no.~42 (2021).
\href{https://doi.org/10.1007/s00012-021-00731-8}{\nolinkurl{doi:10.1007/s00012-021-00731-8}}.

\bibitem{WRZ2024}
Wu, Y.N., Ren, M.M., Zhao, X.Z.:
The additively idempotent semiring $S_7^0$ is nonfinitely based.
Semigroup Forum \textbf{108}(2), 479--487 (2024).
\href{https://doi.org/10.1007/s00233-024-10420-2}{\nolinkurl{doi:10.1007/s00233-024-10420-2}}.

\bibitem{WZR2023}
Wu, Y.N., Zhao, X.Z., Ren, M.M.:
On varieties of flat nil-semirings.
Semigroup Forum \textbf{106}(1), 271--284 (2023).
\href{https://doi.org/10.1007/s00233-023-10337-2}{\nolinkurl{doi:10.1007/s00233-023-10337-2}}.

\bibitem{YRG}
Yue, M.Y., Ren, M.M., Gao, Z.D.:
Two nonfinitely-based additively idempotent semirings of order four.
Int. J. Algebra Comput. (2026), published online.
\href{https://doi.org/10.1142/S0218196726500414}{\nolinkurl{doi:10.1142/S0218196726500414}}.

\bibitem{YRZS2025}
Yue, M.Y., Ren, M.M., Zeng, L.L., Shao, Y.:
The finite basis problem for additively idempotent semirings of order four, II.
Algebra Univers. \textbf{86}, article no.~33 (2025).
\href{https://doi.org/10.1007/s00012-025-00908-5}{\nolinkurl{doi:10.1007/s00012-025-00908-5}}.

\bibitem{ZRCS2020}
Zhao, X.Z., Ren, M.M., Crvenkovi\'{c}, S., Shao, Y., Dapi\'{c}, P.:
The variety generated by an ai-semiring of order three.
Ural Math. J. \textbf{6}(2), 117--132 (2020).
\href{https://doi.org/10.15826/umj.2020.2.012}{\nolinkurl{doi:10.15826/umj.2020.2.012}}.

\bibitem{ZW2025}
Zhao, X.Z., Wu, Y.N.:
The finite basis problem for ai-semirings obtained from a commutative flat semiring by adjoining a zero.
Commun. Algebra \textbf{53}(9), 3721--3732 (2025).
\href{https://doi.org/10.1080/00927872.2025.2467188}{\nolinkurl{doi:10.1080/00927872.2025.2467188}}.
\end{thebibliography}
\end{document}